\documentclass[11pt]{amsart}
\usepackage{a4wide,wasysym,amssymb,indentfirst,graphicx,cite,amsthm,color}
\usepackage[bookmarksnumbered, colorlinks, linktocpage, plainpages]{hyperref}
\usepackage{epsfig}
\usepackage{float}
\usepackage{mathdots}
\usepackage{extarrows}
\usepackage{indentfirst}
 \usepackage{yhmath}
\newtheorem{theorem}{Theorem}[section]

\newtheorem{lemma}[theorem]{Lemma}
\newtheorem{proposition}[theorem]{Proposition}
\theoremstyle{definition}
\newtheorem{definition}[theorem]{Definition}
\newcommand{\diam}{{\rm diam}\,}
\newcommand{\dist}{{\rm dist}\,}
\newcommand{\mcl}{{\mathcal{L}}}
\newcommand{\mbr}{{\mathbb{R}}}
\newcommand{\mbs}{{\mathbb{S}}}
\newcommand{\mbd}{{\mathbb{D}}}

\theoremstyle{example}
\newtheorem{example}[theorem]{Example}
\theoremstyle{remark}
\newtheorem{remark}[theorem]{Remark}
\numberwithin{equation}{section}
\makeatletter
\newcommand*{\mint}[1]{%
  \mint@l{#1}{}%
}
\newcommand*{\mint@l}[2]{%
  \@ifnextchar\limits{%
    \mint@l{#1}%
  }{%
    \@ifnextchar\nolimits{%
      \mint@l{#1}%
    }{%
      \@ifnextchar\displaylimits{%
        \mint@l{#1}%
      }{%
        \mint@s{#2}{#1}%
      }%
    }%
  }%
}
\newcommand*{\mint@s}[2]{%
  \@ifnextchar_{%
    \mint@sub{#1}{#2}%
  }{%
    \@ifnextchar^{%
      \mint@sup{#1}{#2}%
    }{%
      \mint@{#1}{#2}{}{}%
    }%
  }%
}
\def\mint@sub#1#2_#3{%
  \@ifnextchar^{%
    \mint@sub@sup{#1}{#2}{#3}%
  }{%
    \mint@{#1}{#2}{#3}{}%
  }%
}
\def\mint@sup#1#2^#3{%
  \@ifnextchar_{%
    \mint@sup@sub{#1}{#2}{#3}%
  }{%
    \mint@{#1}{#2}{}{#3}%
  }%
}
\def\mint@sub@sup#1#2#3^#4{%
  \mint@{#1}{#2}{#3}{#4}%
}
\def\mint@sup@sub#1#2#3_#4{%
  \mint@{#1}{#2}{#4}{#3}%
}
\newcommand*{\mint@}[4]{%
  \mathop{}%
  \mkern-\thinmuskip
  \mathchoice{%
    \mint@@{#1}{#2}{#3}{#4}%
        \displaystyle\textstyle\scriptstyle
  }{%
    \mint@@{#1}{#2}{#3}{#4}%
        \textstyle\scriptstyle\scriptstyle
  }{%
    \mint@@{#1}{#2}{#3}{#4}%
        \scriptstyle\scriptscriptstyle\scriptscriptstyle
  }{%
    \mint@@{#1}{#2}{#3}{#4}%
        \scriptscriptstyle\scriptscriptstyle\scriptscriptstyle
  }%
  \mkern-\thinmuskip
  \int#1%
  \ifx\\#3\\\else_{#3}\fi
  \ifx\\#4\\\else^{#4}\fi
}
\newcommand*{\mint@@}[7]{%
  \begingroup
    \sbox0{$#5\int\m@th$}%
    \sbox2{$#5\int_{}\m@th$}%
    \dimen2=\wd0 %
    \let\mint@limits=#1\relax
    \ifx\mint@limits\relax
      \sbox4{$#5\int_{\kern1sp}^{\kern1sp}\m@th$}%
      \ifdim\wd4>\wd2 %
        \let\mint@limits=\nolimits
      \else
        \let\mint@limits=\limits
      \fi
    \fi
    \ifx\mint@limits\displaylimits
      \ifx#5\displaystyle
        \let\mint@limits=\limits
      \fi
    \fi
    \ifx\mint@limits\limits
      \sbox0{$#7#3\m@th$}%
      \sbox2{$#7#4\m@th$}%
      \ifdim\wd0>\dimen2 %
        \dimen2=\wd0 %
      \fi
      \ifdim\wd2>\dimen2 %
        \dimen2=\wd2 %
      \fi
    \fi
    \rlap{%
      $#5%
        \vcenter{%
          \hbox to\dimen2{%
            \hss
            $#6{#2}\m@th$%
            \hss
          }%
        }%
      $%
    }%
  \endgroup
}
\begin{document}
\title{\bf Weighted estimates for diffeomorphic extensions of homeomorphisms}
\author{Haiqing Xu}
\date {\today}

\begin{abstract}
Let $\Omega \subset \mbr^2$ be an internal chord-arc domain and $\varphi : \mbs^1 \rightarrow \partial \Omega$ be a homeomorphism.
Then there is a diffeomorphic extension $h : \mbd \rightarrow \Omega$ of $\varphi .$ We study the relationship between weighted integrability of the derivatives of $h$ and double integrals of $\varphi$ and of $\varphi^{-1} .$ 

\end{abstract}

\maketitle

\section{Introduction}
Let $\Omega\subset \mbr^2$ be a bounded convex domain. Suppose that $\varphi$ is a homeomorphism from the unit circle $\mathbb{S}^1$ onto $\partial \Omega.$ Then, by Rad\'o \cite{Rado 1926 Jahresber. Deutsch. Math.-Verein.}, 
Kneser \cite{Kneser 1926 Jahresber. Deutsch. Math.-Verein.},
Choquet \cite{Choquet 1945 Bull. Sci. Math.} and
Lewy \cite{Lewy 1936 Bull. Amer. Math. Soc.},
the complex-valued Poisson extension $h$ of $\varphi$ is a diffeomorphism from $\mathbb{D}$ onto $\Omega$. 
We are interested in the integrability degrees of the derivatives of $h.$
In $2007$, G. C. Verchota \cite{Verchota 2007 Proc. Amer. Math. Soc.} proved that the derivatives of $h$ may fail to be square integrable but that they are necessarily $p$-integrable over $\mathbb{D}$ for all $p<2$. In $2009$, T. Iwaniec, G. J. Martin and C. Sbordone improved on \cite{Iwaniec 2009 Discrete Contin. Dyn. Syst. Ser.} by showing that the derivatives belong to weak-$L^2$ with sharp estimates. 
Actually
\begin{equation}\label{motiv1}
\int_{\mathbb{D}} |Dh (z)|^2 \, dz \approx \int_{\mbs^1} \int_{\mbs^1} \frac{|\varphi(\xi) -\varphi(\eta)|^2}{|\xi -\eta|^2} \, |d \xi| \, |d \eta|,
\end{equation}
since harmonic functions minimize the $L^2$-energy and the right-hand side of \eqref{motiv1} is the trace norm of $\dot{W} ^{1,2} (\mbd).$
In \cite{Astala 2005 Proc. London Math. Soc.}, it was further shown that if additionally $\partial \Omega$ is a $C^1$-regular Jordan curve then
\begin{equation}\label{int1}
\int_{\mathbb{D}} |Dh(z)|^2 \, dz <\infty
\Leftrightarrow \int_{\partial \Omega} \int_{\partial \Omega} |\log|\varphi^{-1}(\xi) -\varphi^{-1}(\eta)|| \, |d\xi|\, |d\eta|<\infty.
\end{equation}
All the above results require the target domain to be convex.

If $\Omega$ is a bounded non-convex Jordan domain, then there exists a homeomorphism $\varphi : \mbs^1 \rightarrow \partial \Omega$ for which the harmonic extension fails to map $\mbd$ homeomorphically onto $\Omega,$ see \cite{Kneser 1926 Jahresber. Deutsch. Math.-Verein., Choquet 1945 Bull. Sci. Math.}.
Hence we cannot use the harmonic extension to produce a diffeomorphic extension.
Nevertheless, (weighted) analogs of the results as \eqref{int1} for diffeomorphic extensions in the case of an internal chord-arc Jordan domain exist, see \cite{our paper}.
For the definition of (internal) chord-arc domains, we refer to Definition \ref{def int chord-arc}.
Notice that each bounded convex Jordan domain is a chord-arc domain.
In this paper, we generalize the results in \cite{our paper} to the weighted $L^p$-setting.

Let $\Omega$ be an internal chord-arc Jordan domain with the internal distance $\lambda_{\Omega} .$ Assume that $h :\mbd \rightarrow \Omega$ is a diffeomorphism and $\varphi : \mbs^1 \rightarrow \partial \Omega$ is a homeomorphism.
Set $\delta (z) = 1-|z| .$
Given $p >1, \alpha \in \mathbb{R}, \lambda \in \mathbb{R} ,$ 
we define
\begin{equation*}\label{I_1}
I_1 (p,\alpha ,\lambda ,h) = \int_{\mathbb{D}} |Dh (z)|^p \delta^{\alpha} (z) \log^{\lambda} (2 \delta^{-1} (z))\, dz ,
\end{equation*}
\begin{equation*}\label{I_2}
I_2 (p,\alpha ,\lambda ,h) = \int_{\mathbb{D}} |Dh (z)|^p \log^{\lambda} (e+|Dh (z)|) \delta^{\alpha} (z) \, dz ,
\end{equation*}
\begin{equation*}
\mathcal{U} (p,\alpha,\lambda,\varphi)=\int_{\mathbb{S}^1} \int_{\mathbb{S}^1} \frac{\lambda^p _{\Omega}(\varphi(\xi) ,\varphi(\eta))}{|\xi -\eta|^{p-\alpha}} \log^{\lambda} \big( e+\frac{\lambda _{\Omega}(\varphi(\xi) ,\varphi(\eta))}{|\xi -\eta|}\big) \, |d \eta|\, |d \xi| ,
\end{equation*}
\begin{equation*}
\mathcal{A}_{p,\alpha,\lambda} (t) = \int_{1} ^{t}  - x^{1+\alpha-p} \log^{\lambda} _{2} (x^{-1}) \, dx  \qquad \forall t \ge 0, 
\end{equation*}
\begin{equation*}
\mathcal{V} (p,\alpha,\lambda,\varphi)= \int_{\partial \Omega}\big(\int_{\partial \Omega} \mathcal{A}_{p,\alpha ,\lambda} (|\varphi^{-1}(\xi)- \varphi^{-1}(\eta)|)\,  |d \eta| \big)^{p-1} \, |d \xi| .
\end{equation*}
Our main result is the following theorem.

\begin{theorem}\label{main thm}
Let $\Omega\subset\mbr^2$ be an internal chord-arc Jordan domain and $\varphi: \mathbb{S}^1 \rightarrow \partial\Omega$ be a homeomorphism. There is a diffeomorphic extension $h: \mathbb{D} \rightarrow \Omega$ of $\varphi $ for which, for any $p>1 ,$ we have that
\begin{enumerate}
\item[(1)] if either $\alpha \in (p-2, +\infty)$ and $\lambda \in \mathbb{R}$ or $\alpha=p-2$ and $\lambda \in (-\infty, -1),$
\begin{equation*}
\mbox{both }I_1 (p,\alpha ,\lambda ,h) \mbox{ and } I_2 (p,\alpha ,\lambda ,h) \mbox{ are finite.}
\end{equation*}
\item[(2)] if either $\alpha \in (-1,p-2)$ and $\lambda \in \mathbb{R}$ or $\alpha=p-2$ and $\lambda \in [-1,+\infty),$  
\begin{equation*}
\mbox{both }I_1 (p,\alpha ,\lambda ,h) \mbox{ and } I_2 (p,\alpha ,\lambda ,h) \mbox{ are comparable to } \mathcal{U} (p,\alpha,\lambda,\varphi).
\end{equation*}
Moreover whenever $p \in (1,2]$ 
\begin{equation*}
\mbox{both }I_1 (p,\alpha ,\lambda ,h) \mbox{ and }I_2 (p,\alpha ,\lambda ,h)\mbox{ dominate }\mathcal{V} (p,\alpha,\lambda,\varphi),
\end{equation*}
while 
\begin{equation*}
\mathcal{V} (p,\alpha,\lambda,\varphi)\mbox{ controls both }I_1 (p,\alpha ,\lambda ,h)\mbox{ and }I_2 (p,\alpha ,\lambda ,h)
\end{equation*}
for all $p \in [2,+\infty) .$ 
Furthermore both $I_1 (p,\alpha ,\lambda ,h)$ and $I_2 (p,\alpha ,\lambda ,h)$ are in general comparable to $\mathcal{V} (p,\alpha,\lambda,\varphi)$ only for $p =2 .$
\end{enumerate}
For any $p>1,$ there is no a homeomorphic extension $h : \mbd \rightarrow \Omega$ of $\varphi$ for which $I_1 (p,\alpha ,\lambda ,h)<+\infty$ for either $\alpha \in (-\infty, -1)$ and $\lambda \in \mathbb{R}$ or $\alpha=-1$ and $\lambda \in [-1,+\infty);$ and for which $I_2 (p,\alpha ,\lambda ,h)<+\infty$ for all $\alpha \in (-\infty, -1]$ and each $\lambda \in \mathbb{R}.$
\end{theorem}

Motivated by \eqref{int1}, one could hope to use $\mathcal{V}(p,\alpha ,\lambda ,\varphi)$ to control both $I_1 (p,\alpha ,\lambda, h)$ and $I_2 (p,\alpha ,\lambda, h) .$ Example \ref{example 1} together with Example \ref{example 2} shows that $\mathcal{V}(p,\alpha ,\lambda ,\varphi)$
is comparable to $I_1 (p,\alpha ,\lambda, h)$ or $I_2 (p,\alpha ,\lambda, h) $ only when $p =2 .$
Theorem \ref{main thm} does not cover the case where $p>1$, $\alpha =-1$ and $\lambda \in (-\infty ,-1 ) .$ We will return to this case in a future paper.

The structure of this paper is the following. In the next section, we give some preliminaries. Section $3$ is the proof of Theorem \ref{main thm}. The final section contains several examples related to Theorem \ref{main thm} ($2$).

\section{Preliminaries}

By $s \gg  1$ and $t \ll  1$ we mean that $s$ is sufficiently large and $t$ is sufficiently small, respectively.
By $f \lesssim  g$ we mean that there exists a constant $C > 0$ such that $f(x) \le  Cg(x)$ for every $x$. If $f \lesssim  g$ and $g \lesssim  f$ we may denote $f \approx g$.
By $\mathbb{N}$ and $\mathbb{R}$ we denote the set of all positive integers and the set of all real numbers.
Let $\mathcal{L}^2$ (respectively $\mathcal{L}^1$) be the $2$-dimensional ($1$-dimensional) Lebesgue measure.
For sets $E \in \mathbb{R}^2$ and $F \in \mathbb{R}^2 ,$ let $\diam (E)$ be the diameter of $E,$ and $\dist (E,F)$ be the Euclidean distance between $E$ and $F .$
Let $B(p,r)$ be the disk with center $P$ and radius $r .$

\begin{definition}\label{def int chord-arc}
A Jordan domain $\Omega\subset\mbr^2$ is an {\bf{internal chord-arc Jordan domain}} if $\partial \Omega$ is rectifiable and there is a constant $C>0$ such that for all $w_1, w_2\in \partial\Omega$,
\begin{equation}\label{e1.1}
{\ell(w_1, w_2)}\leq C \lambda_\Omega(w_1, w_2),
\end{equation}
where $\ell(w_1, w_2)$ is the arc length of the shorter arc of $\partial\Omega$ joining $w_1$ to $w_2$, and $\lambda_\Omega(w_1, w_2)$ is the {\bf{internal distance}} between $w_1, w_2$, which is defined as
 $${\lambda_\Omega(w_1, w_2)}=\inf_\alpha \ell(\alpha),$$
where the infimum is taken over all rectifiable arcs $\alpha \subset  \Omega$ joining $w_1$ and $w_2$; if there is no rectifiable curve joining $w_1$ and $w_2$, we set ${\lambda_\Omega(w_1, w_2)}=\infty$; cf. \cite[Section 3.1]{NV91} or \cite [Section 2]{OJB06}. 
\end{definition}
If \eqref{e1.1} holds for the Euclidean distance instead of the internal distance, we call $\Omega$ be a {\bf{chord-arc domain}}.
Naturally, every chord-arc Jordan domain is an internal chord-arc domain, but there are internal chord-arc domains that fail to be chord-arc; e.g. the standard cardioid domain
\begin{equation*}\label{Delta}
\Delta = \{(x,y) \in \mathbb{R}^2 : (x^2 +y^2)^2 -4x (x^2 +y^2) -4y^2 <0\} .
\end{equation*}

\subsection{Dyadic decomposition}\label{Dyadic decomposition}

Given $j\in \mathbb N$ and $k=1,..., 2^j,$ let 
\begin{equation}\label{dyadic deom 1}
I_{j, k}=[2 \pi (k-1) 2^{-j}, 2 \pi k 2^{-j}] , \ \Gamma_{j,k} = \{e^{i \theta} : \theta \in I_{j, k} \}.
\end{equation}
Then $\{I_{j, k}\}$ is a dyadic decomposition of $[0, 2\pi] $ and $\{\Gamma_{j,k}\}$ is a dyadic decomposition of $\mbs^1 .$ 
We call $\Gamma_{j,k}$ a $j$-level dyadic arc. 
Moreover we have that
\begin{equation}\label{dyadic deom 1-0}
\ell (\Gamma_{j,k}) \approx 2^{-j} \qquad \forall j \in \mathbb{N} \mbox{ and } k=1,...,2^j .
\end{equation}
Based on \eqref{dyadic deom 1}, 
there is a decomposition of the unit disk $\mathbb{D}$ given by  $\{Q_{j,k}: j\in \mathbb N \mbox{ and } k=1,..., 2^j\} ,$ where
\begin{equation}\label{dyadic deom 1-1}
Q_{j, k}=\left\{re^{i\theta}: 1-2^{1-j} \le r \le  1-2^{-j} \mbox{ and } \theta \in I_{j, k}\right\}.
\end{equation}
By \eqref{dyadic deom 1-0} it follows that
\begin{equation}\label{dyadic deom 1-2}
\mcl^2(Q_{j,k}) \approx 2^{-2j} \approx \ell(\Gamma_{j,k})^2 \qquad \forall j \in \mathbb{N} \mbox{ and } k= 1,..., 2^j.
\end{equation}
Moreover there is a uniform constant $C>0$ such that for any $Q_{j,k}$ 
there is a disk $B_{j,k}$ satisfying 
\begin{equation}\label{dyadic deom 2}
B_{j,k} \subset Q_{j,k} \subset CB_{j,k} .
\end{equation}

\subsection{$A_p$ weights}

\begin{definition}
For a given $p \in (1,+\infty),$ a locally integrable function $w :\mathbb{R}^2 \rightarrow [0,+\infty)$ is an $A_p$ weight if there is a constant $C>0$ such that for any disk $B \subset \mathbb{R}^2$ we have that
\begin{equation*}
\frac{1}{\mcl^2(B)} \int_{B} w(x) \, dx \le C \big( \frac{1}{\mcl^2(B)} \int_{B} w(x) ^{\frac{1}{1-p}} \, dx \big)^{1-p}.
\end{equation*}
Next, $w$ is an $A_1$ weight if there is a constant $C>0$ such that
\begin{equation*}
\frac{1}{\mcl^2(B)} \int_{B} w(z) \, dz \le C w(x)
\end{equation*}
for each disk $B \subset \mathbb{R}^2 $ and all $x \in B .$ 
\end{definition}

For more information on $A_p$ weights, we recommend \cite{Coifman 1974 Studia Math., Muckenhoupt 1972 Trans. Amer. Math. Soc., Jones 1980 Ann. of Math.}.
Let $\delta (x) = \mbox{dist} (\mathbb{S}^1, x) .$
Given $\alpha \in (-1,p-1)$ and $\lambda \in \mathbb{R},$ we define
\begin{equation}\label{A_p weight}
w_{\alpha ,\lambda}(x) =
\begin{cases}
\delta(x) ^{\alpha} \log^{\lambda} \left( 2 \delta^{-1} (x)\right) & 0\le |x| \le 2 ,\\
\log^{\lambda} (2) & |x| \ge 2 .
\end{cases}
\end{equation}
It is well known that $w_{\alpha ,0}$ belongs to $A_p.$ We now generalize this to all $\lambda \in \mathbb{R}.$

\begin{proposition}\label{proposition: A_p weight}
Let $p \ge 1 $ and $w_{\alpha ,\lambda}$ be as in \eqref{A_p weight}.
Then $w_{\alpha ,\lambda}$ is an $A_p$ weight for all $\alpha \in (-1 , p-1)$ and $\lambda \in \mathbb{R}.$
\end{proposition}

\begin{proof}
The idea of proof is to use the Jones factorization of $A_p$ weights (see \cite{Jones 1980 Ann. of Math.}), i.e. we should prove $w_{\alpha ,\lambda}= w_1 w_{2} ^{1-p}$ for two $A_1$ weights $w_1$ and $w_2.$

We first consider the case $\lambda \ge 0 .$
For a given $\alpha \in (-1, p-1),$ there unique exist $a_1 \in (0,1)$ and $a_2 \in (0,1)$ such that $\alpha = a_1 (-1) + a_2 (p-1).$
Set $\alpha_1 = -a_1$, $\alpha_2 = -a_2$, $\lambda_1 =p \lambda$ and $\lambda_2 = \lambda.$ 
We define 
 \begin{equation}\label{w_1 definition}
w_{1}(x) =
\begin{cases}
\delta(x) ^{\alpha_1} \log^{\lambda_1} \left( 2 \delta^{-1}(x) \right) & 0\le |x| \le 2 , \\
\log^{\lambda_1} (2) & |x| \ge 2 ,
\end{cases}
\end{equation}
and
\begin{equation}\label{w_2 definition}
w_{2}(x) =
\begin{cases}
\delta (x) ^{\alpha_2} \log^{\lambda_2} \left( 2 \delta^{-1}(x) \right) & 0\le |x| \le 2 ,\\
\log^{\lambda_2} (2) & |x| \ge 2.
\end{cases}
\end{equation}
We next prove that $w_1$ is an $A_1$ weight, i.e.
\begin{equation}\label{w_1 is A_1}
\displaystyle  \mint{-}_{B} w_1 (x) dx \lesssim \underset{x \in B}{\mbox{inf}}w_1 (x)
\end{equation}
for all disk $B \subset \mathbb{R}^2.$ Let $d_B =\mbox{dist} (B, \mathbb{S}^1) .$ 

\textsf{Case 1:} $d_B \ge \mbox{diam}(B) /2.$ 
We have that 
\begin{equation}\label{case 1-1}
d_B \le \delta(x) \le 3 d_B \qquad \forall x \in B. 
\end{equation}
If $1 \le d_B,$ then $\delta(x) \ge 1$ for all $x \in B.$ Therefore $w_1 (x) = \log^{\lambda_1} (2)$ whenever $x \in B.$ Of course \eqref{w_1 is A_1} holds now.
If $3d_B \le 1,$ then $w_1 (x) = \delta (x) ^{\alpha_1} \log^{\lambda_1} \left( 2 \delta^{-1} (x)\right)$ for all $x \in B.$
By \eqref{case 1-1} it hence follows that
$w_1 (x) \approx d^{\alpha_1} _B \log^{\lambda_1} \left( 2 d^{-1} _B \right)$ whenever $ x \in B.$ Therefore \eqref{w_1 is A_1} holds.
If $d_B < 1 < 3d_B ,$ 
let $B_{1} =\{x \in B: d_B < \delta(x) < 1 \}$ and $B_{2} =\{x \in B: 1 \le \delta(x) < 3d_B \}.$
Then $B= B_1 \cup B_2 $ and 
\begin{equation}\label{case 1-2}
w_1 (x) = \log^{\lambda_1} (2) \qquad \mbox{whenever } x \in B_2.
\end{equation}
Since 
\begin{equation}\label{case 1-2-11}
\left[t^{\alpha_1} \log^{\lambda_1} \left( 2 t^{-1} \right) \right]' = t^{\alpha_1 -1} \log^{\lambda_1} ( 2t^{-1} ) \left(\alpha_1 -  \frac{\lambda_1}{\log ( 2 t^{-1} )} \right) <0,
\end{equation}
for all $t \in (0,1],$ we have that
\begin{equation}\label{case 1-3}
w_1 (x) \le d^{\alpha_1} _B \log^{\lambda_1} (2 d^{-1} _B) \le \frac{\log^{\lambda_1} (6)}{3^{\alpha_1}} \qquad \forall x \in B_1.
\end{equation}
Combining \eqref{case 1-2} and \eqref{case 1-3} implies that
\begin{align*}
\displaystyle  \mint{-}_{B} w_1 (x) dx & = \frac{1}{\mcl^2(B)} \left(\int_{B_1} w_1 + \int_{B_2} w_1 \right) \\
& \le \frac{1}{\mcl^2(B)} \left(\mcl^2(B_1) \frac{\log^{\lambda_1} (6)}{3^{\alpha_1}} + \mcl^2(B_2) \log^{\lambda_1} (2) \right) \\
& \lesssim \log^{\lambda_1} (2) = \underset{x \in B}{\mbox{inf}} w_1 (x).
\end{align*}

\textsf{Case 2:} $d_B < \mbox{diam} (B)/2$ and $\mbox{diam}(B) \le 2/3.$
Pick $x' \in \partial B$ and $x_0 \in \mathbb{S}^1$ such that $\mbox{dist} (B, \mathbb{S}^1) = |x'-x_0| .$
Let $r_B = 3 \mbox{diam}(B)/2.$ 
Since 
\begin{equation*}
|x - x_0| \le |x-x'| + |x' - x_0| \le r_B 
\end{equation*}
for all $x \in B ,$ we have $B \subset B(x_0, r_B).$
Let
$E =\{x \in \mathbb{R}^2: \mbox{dist} (x, \mathbb{S}^1) < r_B \} .$ Then $B(x_0, r_B) \subset E.$
Since $\mcl^2(B(x_0, r_B)) = \pi r^2 _B$ and $\mcl^2(E) = 4 \pi r_B,$ the maximal number of pairwise disjoint open disks $B(x,r_B)$ with $x \in \mathbb{S}^1$ is less than $4 r^{-1} _B.$
We have that
\begin{align}\label{case 2:1}
\frac{1}{\mcl^2(B)} \int_{B} w_1 (x) \, dx
\le & \frac{1}{\mcl^2(B)} \int_{B(x_0, r_B)} w_1 (x) \, dx \notag \\
\lesssim & \frac{r_B}{\mcl^2(B)} \int_{E} w_1 (x) \, dx
\approx \frac{1}{r_B} \int_{0} ^{r_B} t^{\alpha_1} \log^{\lambda_1} \left( 2 t^{-1} \right) \, dt.
\end{align}
Notice that
\begin{equation}\label{case 2:1-1}
\left[t^{\alpha_1 +1} \log^{\lambda_1} \left(2 t^{-1} \right)\right]'  = t^{\alpha_1 } \log^{\lambda_1} ( 2 t^{-1} ) \left(\alpha_1 +1 - \frac{\lambda_1}{\log (2 t^{-1} )} \right)\qquad  t >0.
\end{equation}
Since $\underset{t \rightarrow 0^+}{\mbox{lim}} \alpha_1 +1 - \frac{\lambda_1}{\log \left(2 t^{-1} \right)}= \alpha_1 +1$ and $\alpha_1 +1 - \frac{\lambda_1}{\log \left(2 t^{-1} \right)}$ is decreasing with respect to $t >0,$ there exists $\epsilon \in (0,1)$ determined by $\alpha_1$ and $\lambda_1$ such that $\alpha_1 +1 - \frac{\lambda_1}{\log (2 \epsilon^{-1} )} \ge (\alpha_1 +1)/2.$
We then obtain from \eqref{case 2:1-1} that
\begin{equation*}
\left[t^{\alpha_1 +1} \log^{\lambda_1} \left( 2 t^{-1} \right)\right]' \ge \frac{\alpha_1 +1}{2} t^{\alpha_1 } \log^{\lambda_1} (2 t^{-1} )
\end{equation*}
for all $t \in [0,\epsilon r_B].$
Therefore
\begin{equation}\label{case 2:2}
\int_{0} ^{\epsilon r_B} t^{\alpha_1} \log^{\lambda_1} (2 t^{-1} ) \, dt
= \frac{2(\epsilon r_B)^{\alpha_1 +1}}{\alpha_1 +1} \log^{\lambda_1} (2 (\epsilon r_B)^{-1} )
\lesssim r^{\alpha_1 +1} _B \log^{\lambda_1} \left(2 r^{-1} _B \right)
\end{equation}
Moreover by \eqref{case 1-2-11} we have that
\begin{equation}\label{case 2:3}
\int_{\epsilon r_B} ^{r_B} t^{\alpha_1} \log^{\lambda_1} (2 t^{-1} ) \, dt
\le (r_B- \epsilon r_B) (\epsilon r_B)^{\alpha_1} \log^{\lambda_1} (2 (\epsilon r_B )^{-1} )
\lesssim r^{\alpha_1 +1} _B \log^{\lambda_1} (2 r^{-1} _B ).
\end{equation}
Combining \eqref{case 2:1}, \eqref{case 2:2} with \eqref{case 2:3} implies that
\begin{equation*}
\frac{1}{|B|} \int_{B} w_1 (x) \, dx
\lesssim r^{\alpha_1} _B \log^{\lambda_1} \left(2 r^{-1} _B \right).
\end{equation*}
Together with
\begin{equation*}
r^{\alpha_1} _B \log^{\lambda_1} \left(2 r^{-1} _B \right)
= \underset{t \in [0,r_B]}{\mbox{inf}} t^{\alpha_1} \log^{\lambda_1} \left(2 t^{-1} \right)
= \underset{x \in E}{\mbox{inf}}\ w_1 (x)
\le \underset{x \in B}{\mbox{inf}}\ w_1 (x),
\end{equation*}
we hence obtain \eqref{w_1 is A_1}.

\textsf{Case 3:} $d_B < \mbox{diam}(B)/2$ and $\mbox{diam}(B) > 2/3.$
Let $x'$ and $x_0$ be as in Case $2 .$ 
Then $|x'| = 1+ \mbox{dist}(x', \mathbb{S}^1)  \le 1 + \mbox{diam}(B) 2^{-1}.$
Together with the fact that $|x-x'| \le \mbox{diam}(B)$ for all $x \in B,$ we have $B \subset B(0,1+r_B) .$ 
Moreover by \eqref{case 2:2} and \eqref{case 2:3}, we obtain that
\begin{equation*}\label{case 3:1}
\int_{B(0,2)} w_1 (x) \, dx =  \int_{B(0,1)} + \int_{B(0,2) \setminus B(0,1)}
=  4 \pi \int_{0} ^{1} t^{\alpha_1} \log^{\lambda_1} (2 t^{-1}) \, dt \approx 1.
\end{equation*}
Therefore
\begin{align}\label{case 3:2}
\frac{1}{\mcl^2 (B)} \int_{B} w_1 (x) \, dx \lesssim &  \frac{1}{r^2 _B} \left( \int_{B(0,2)} + \int_{B(0,1+r_B) \setminus B(0,2)}\right) \notag \\
\lesssim & \frac{1}{r^2 _B} \left(\mcl^2 (B(0,2)) + \log^{\lambda_1} (2) \mcl^2 (B(0,1+r_B) \setminus B(0,2)) \right) \notag \\
\lesssim & 1 .
\end{align}
Moreover by the monotonicity of $t^{\alpha_1} \log^{\lambda_1} (2 t^{-1})$ on $(0,+\infty) ,$ we have that
\begin{equation}\label{case 3:3}
\log^{\lambda_1} \left(2 \right)
= \underset{t \in [0,1+r_B]}{\mbox{inf}} t^{\alpha_1} \log^{\lambda_1} \left(2 t^{-1} \right)
= \underset{x \in B(0,1+r_B)}{\mbox{inf}}\ w_1 (x)
\le \underset{x \in B}{\mbox{inf}}\ w_1 (x).
\end{equation}
By combining \eqref{case 3:2} with \eqref{case 3:3}, we obtain \eqref{w_1 is A_1}.

By the analogous arguments as for \eqref{w_1 is A_1}, 
we obtain $w_2 \in A_1.$ Therefore the Jones factorization theorem implies that $w_{\alpha, \lambda} \in A_p$ for all $\alpha \in (-1,p-1)$ and $\lambda \ge 0.$

When $\lambda<0,$ define $w_1$ and $w_2$ as in \eqref{w_1 definition} and \eqref{w_2 definition} with $\lambda_1 =-\lambda,\ \lambda_2 = 2 \lambda (1-p)^{-1}$ and both $\alpha_1$ and $\alpha_2$ invariant. 
By the same arguments as for the case $\lambda \ge 0 ,$ we obtain that $w_{\alpha,\lambda} \in A_p$ whenever $\alpha \in (-1,p-1)$ and $\lambda<0.$
\end{proof}

\subsection{A class of functions}

\begin{definition}
We say that a function $f: [0,+\infty) \rightarrow [0,+\infty)$ satisfies the $\Delta_2$-condition if there is a constant $C>0$ such that
\begin{equation*}
f(2 t) \le C f(t)
\end{equation*}
for all $t \in [0,+\infty).$
\end{definition}

Given $p>1$ and $\lambda \in \mathbb{R},$ set
\begin{equation}\label{Phi definition}
\Phi_{p,\lambda}(t) = t^{p} \log^{\lambda} (e+t) \qquad \mbox{for } t \in [0,+\infty).
\end{equation} 

\begin{proposition}\label{proposition of Phi: lamdba >0}
Let $\Phi_{p,\lambda}$ be as in \eqref{Phi definition} with $p >1$ and $\lambda \ge 0 .$ Then
\begin{enumerate}
\item[(P-1)]$\Phi_{p,\lambda}(t)$ is strictly increasing, continuous and convex on $[0,+\infty),$
\item[(P-2)]$\Phi_{p,\lambda}(t)$ satisfies the $\Delta_2$-condition on $[0,+\infty),$
\item[(P-3)]there is a constant $C>0$ such that $\Phi_{p,\lambda}'(t) \le C \Phi_{p,\lambda}(t)/t$ for all $t \in (0,+\infty),$
\item[(P-4)]there is a constant $C>0$ such that $\Phi_{p,\lambda}(st) \le C s^r \Phi_{p,\lambda}(t)$ holds for all $s \in [0,1],\ t \in [0,+\infty)$ and each $r \in (0,p).$
\end{enumerate}
\end{proposition}

\begin{proof}
Simple calculations show that
\begin{equation}\label{Phi'}
\Phi_{p,\lambda}'(t)
=   \left(p \log(e+t)+ \lambda \frac{t}{e+t} \right) t^{p -1} \log^{\lambda - 1}(e+t)
\end{equation}
and
\begin{align}\label{Phi''}
\Phi_{p,\lambda}''(t) = & \left(p(p-1)\log^2 (e+t) + \lambda e \frac{t}{(e+t)^2} \log (e+t) + R_{p,\lambda}(t)  \right) \notag \\
& \times t^{p-2} \log^{\lambda-2} (e+t)  
\end{align}
where $R_{p,\lambda}(t)= \lambda(\lambda-1) (t (e+t)^{-1} )^2  + \lambda(2p-1)  t (e+t)^{-1} \log (e+t) .$
Since $p \log(e+t)+ \lambda t (e+t)^{-1} >0$ for all $t \in (0,+\infty),$ it follows from \eqref{Phi'} that $\Phi_{p,\lambda}'(t) > 0$ for all $t >0.$ Therefore $\Phi_{p,\lambda}$ is strictly increasing on $[0,\infty).$
If $\lambda \ge 1,$ we have that
\begin{equation}\label{func1}
R_{p,\lambda}(t) \ge 0 \qquad \mbox{for all } t \ge 0.
\end{equation}
Whenever $0 \le \lambda <1,$
since $t/(e+t) <1$ and $\log(e+t) \ge 1$ for all $t \ge 0$ we have that
\begin{align}\label{func2}
R_{p, \lambda}(t) = &\frac{t}{e+t} \left(\lambda(\lambda-1) \frac{t}{e+t}  +\lambda(2p-1) \log(e+t) \right) \notag\\
\ge & \frac{t}{e+t} \big(\lambda(\lambda-1) +  \lambda(2p-1) \big) =  \frac{t}{e+t} \left(\lambda^2 + 2\lambda(p-1)\right) \ge 0
\end{align}
for all $t \ge 0.$
By \eqref{Phi''}, \eqref{func1} and \eqref{func2},
we have that $\Phi'' _{p,\lambda} (t) \ge 0$ for all $t \ge 0 .$
Therefore $\Phi_{p,\lambda}$ is convex on $[0,+\infty) .$

Since both $t^p$ and $\log^{\lambda} (e+t)$ satisfy the $\Delta_2$-condition on $[0,+\infty),$ (P-2) then holds.

Since $p \log (e+t) + \lambda t(e+t)^{-1} \le (p+\lambda) \log(e+t)$ for all $t \in [0,+\infty),$ from \eqref{Phi'} we obtain that
\begin{equation*}
\Phi_{p,\lambda}'(t) \le (p+\lambda) t^{p-1} \log ^{\lambda} (e+t) = (p+\lambda) \frac{\Phi_{p,\lambda}(t)}{t} 
\end{equation*}
for all $t>0 .$
Hence (P-3) holds.

In order to prove (P-4), it suffices to prove that $s^p \log^{\lambda}(e+st) \le s^r \log^{\lambda}(e+t)$ for all $s \in [0,1],\ t \in [0,+\infty)$ and each $r \in (0,p).$
In fact, for any $0 \le s \le 1$ and $t \ge 0,$ we have that $\log^{\lambda}(e+st) \le \log^{\lambda}(e+t).$
Hence for any $r \in (0,p),$ it follows that
$s^p \log^{\lambda}(e+st) \le s^r \log^{\lambda}(e+t)$ for all $s \in [0,1]$ and $t \in [0,+\infty) .$
\end{proof}

We define the Hardy-Littlewood maximal function for a Lebesgue measurable function $f$ in $\mathbb{R}^2$ as
\begin{equation*}
M_f (x) = \sup_{x \in B} \displaystyle  \mint{-}_{B} |f(z)| dz = \sup_{x \in B} \frac{1}{|B|} \int_{B} |f(z)| \, dz
\end{equation*}
where the supremum is taken over all disks $B \subset \mathbb{R}^2$ containing $x.$
The following lemma shows the value of Proposition \ref{proposition of Phi: lamdba >0}.

\begin{lemma}\label{weighted muckenhoupt}
For $p >1,$ let $w$ be an $A_p$ weight and $\Phi_{p,\lambda}(t)$ be as in \eqref{Phi definition} with $\lambda \ge 0.$ Given a Lebesgue measurable function $f,$ we have that
\begin{equation*}
\int_{\mathbb{R}^2} \Phi_{p,\lambda}(M_f (x)) w(x) \, dx \lesssim \int_{\mathbb{R}^2}  \Phi_{p,\lambda}(|f(x)|) w(x) \, dx .
\end{equation*}
\end{lemma}

\begin{proof}
By the open-end property of $A_p$ weights, we have that $w$ is an $ A_{r_0}$ weight for some $r_o < p.$
For $t \in [0,+\infty) ,$ 
set $g_t (x) = |f(x)| \chi_{\{x \in \mbr^2 : 2 |f(x)| >t \}} .$ 
Muckenhoupt's theorem implies that
\begin{equation}\label{weighted muckenhoupt < 1}
\int_{\mathbb{R}^2} M^{r_0} _{g_t} (x) w(x) \, dx \lesssim \int_{\mathbb{R}^n} |g_t (x)|^{r_0} w(x) \, dx = \int_{\{x \in \mbr^2 : 2 |f(x)| >t \}} |f (x)|^{r_0} w(x) \, dx
\end{equation}
for all $t \ge 0 .$
Moreover, by Chebyshev's inequality we obtain that
\begin{align}\label{weighted muckenhoupt < 2}
\left(\frac{t}{2} \right)^{r_0} \int_{\{x \in \mbr^2 :  2 M_{g_t} (x) > t\}} w(x) \, dx
\le & \int_{\{x \in \mbr^2: 2 M_{g_t} (x) > t \}} M^{r_0} _{g_t} (x) w(x) \, dx \notag \\
\le & \int_{\mathbb{R}^2} M^{r_0} _{g_t}(x) w(x) \, dx.
\end{align}
Since $\{x \in \mbr^2: M_f (x) >t\} \subset \{x \in \mbr^2:  2 M_{g_t} (x) >t \}$ for all $t >0,$ we obtain that
\begin{equation}\label{weighted muckenhoupt < 3}
\int_{\{x \in \mbr^2 : M_f (x) >t\}} w(x) \, dx
\le \int_{\{x \in \mbr^2 : 2 M_{g_t} (x) >t \}} w(x) \, dx.
\end{equation}
Combining \eqref{weighted muckenhoupt < 1}, \eqref{weighted muckenhoupt < 2} with \eqref{weighted muckenhoupt < 3} gives that
\begin{equation}\label{weighted muckenhoupt < 4}
\int_{\{x \in \mbr^2 : M_f (x) >t\}} w(x) \, dx \lesssim t^{-r_0} \int_{\{x \in \mbr^2: 2 |f(x)| >t\}} |f (x)|^{r_0} w(x) \, dx
\end{equation}
for all $t >0.$
Notice that $\Phi_{p,\lambda}(0) =0.$
By Fubini's theorem and (P-3) in Proposition \ref{proposition of Phi: lamdba >0}, we derive from \eqref{weighted muckenhoupt < 4} that
\begin{align}\label{weighted muckenhoupt < 5}
\int_{\mathbb{R}^2} \Phi_{p,\lambda}(M_f (x)) w(x) \, dx
= & \int_{\mathbb{R}^2} \int_{0} ^{+\infty} \chi_{\{x \in \mbr^2: M_f (x) >t \}} w(x)  \, d \Phi_{p,\lambda}(t) \, dx \notag \\
= & \int_{0} ^{+\infty} \int_{\{x \in \mbr^2: M_f (x) >t \}} w(x)\, dx \, d \Phi_{p,\lambda} (t) \notag \\
\lesssim &  \int_{0} ^{+\infty} \frac{\Phi_{p,\lambda}(t)}{t^{1+r_0}} \int_{\{x \in \mbr^2 : 2|f(x)| >t\}} |f (x)|^{r_0} w(x)\, dx \, dt.
\end{align}
Moreover by Fubini's theorem, a change of variables, and (P-2) and (P-4) in Proposition \ref{proposition of Phi: lamdba >0}, there is $r \in (r_0, p)$ such that
\begin{align}\label{weighted muckenhoupt < 6}
\int_{0} ^{+\infty} \frac{\Phi_{p,\lambda}(t)}{t^{1 +r_0}} \int _{\{x \in \mbr^2:2|f(x)|>t \}} |f(x)|^{r_0} w(x)\, dx \, dt = & \int_{\mathbb{R}^2} |f(x)|^{r_0} w(x) \int_{0} ^{2|f(x)|} \frac{\Phi_{p,\lambda}(t)}{t^{1+r_0}} \, dt \, dx  \notag \\
= & \frac{1}{2^{r_0}} \int_{\mathbb{R}^2}  w(x) \int_{0} ^{1} \frac{\Phi_{p,\lambda}(2s |f(x)|)}{s^{1+r_0}} \, ds \, dx \notag \\
\lesssim & \int_{\mathbb{R}^2}  \Phi_{p,\lambda}(|f(x)|) w(x)  \int_{0} ^{1} \frac{1}{s^{1+r_0 -r }} \, ds \, dx \notag \\
\approx & \int_{\mathbb{R}^2}  \Phi_{p,\lambda}(|f(x)|) w(x) \, dx.
\end{align}
Combining \eqref{weighted muckenhoupt < 5} with \eqref{weighted muckenhoupt < 6} implies that
\begin{equation*}
\int_{\mathbb{R}^2} \Phi_{p,\lambda}(M_f (x)) w(x) \, dx \lesssim \int_{\mathbb{R}^2}  \Phi_{p,\lambda}(|f(x)|) w(x) \, dx.
\end{equation*}
\end{proof}

Let $\Phi_{p,\lambda}$ be as in \eqref{Phi definition} with $p >1$ and $\lambda <0.$ By \eqref{Phi'} and \eqref{Phi''}, we have that both monotonicity and convexity of $\Phi_{p,\lambda}$ may fail whenever $t \ll 1,$ but still hold for all $t \gg 1.$ We modify $\Phi_{p,\lambda}$ in a neighborhood of the origin so as to ensure (P-1)-(P-4) of Proposition \ref{proposition of Phi: lamdba >0}.

Since $ 2^{-1}(p+1 ) \log(e+t) \le p \log(e+t) + \lambda t(e+t)^{-1} $ whenever $t \gg 1,$ by \eqref{Phi'} there is a constant $t_2 \gg 1$ such that
\begin{equation}\label{(p+1)Phi(t)/2t}
\frac{(p+1)\Phi_{p,\lambda}(t)}{2t} =\frac{p+1}{2} t^{p-1} \log^{\lambda}(e+t) \le \Phi_{p,\lambda}'(t) 
\end{equation}
for all $t \ge t_2.$
Without loss of generality, we assume that
$\Phi_{p,\lambda}$ is strictly increasing and convex on $[t_2 , \infty).$
Since $p t_2 t^{p-1} +t^p \le (p+1) t_2 t^{p-1} \le t^{p}_2 \log^{\lambda}(e+t_2)$ for any $t \ll 1,$ we have that
\begin{equation}\label{t_1 :1}
pt^{p-1} (t_2 - t) \le \Phi_{p,\lambda}(t_2) -t^p 
\end{equation}
for all $t \ll 1.$
Moreover when $t \le t_2 (p-1)/(p+1),$ we have that
\begin{equation}\label{t_1 :2}
\frac{\Phi_{p,\lambda}(t_2) -t^p}{t_2 - t} \le \frac{\Phi_{p,\lambda}(t_2) }{t_2 - t} \le \frac{(p+1)\Phi_{p,\lambda}(t_2)}{2t_2}.
\end{equation}
Therefore by \eqref{(p+1)Phi(t)/2t}, \eqref{t_1 :1} and \eqref{t_1 :2},
there exists a constant $t_1 \ll 1$ such that
\begin{equation*}
pt^{p-1} _1 \le \frac{\Phi_{p,\lambda}(t_2) -t^p _1 }{t_2 - t_1} \le \frac{(p+1)\Phi_{p,\lambda}(t_2)}{2t_2} \le \Phi_{p,\lambda}'(t_2) .
\end{equation*}
Let $k = (\Phi_{p,\lambda}(t_2) -t^p _1)/ (t_2 - t_1).$
Given $p >1$ and $\lambda <0,$ we define
\begin{equation}\label{Phi_M definition}
\Psi_{p,\lambda} (t) =
\begin{cases}
t^{p} &0 \le t < t_1 ,\\
k(t-t_1) +t^{p}_1 & t_1 \le t <t_2 ,\\
\Phi_{p,\lambda}(t) & t_2 \le t.
\end{cases}
\end{equation}

\begin{proposition}\label{Psi_p,lambda}
The function $\Psi_{p,\lambda}$ enjoys the four properties in Proposition \ref{proposition of Phi: lamdba >0}.
\end{proposition}

\begin{proof}
It is easy to see that $\Psi_{p,\lambda}$ is strictly increasing, continuous and convex on $[0,+\infty).$
In order to prove (P-2) for $\Psi_{p,\lambda}$, it suffices to prove that
\begin{equation}\label{enought prove}
\Psi_{p,\lambda}(2t) \le C \Psi_{p,\lambda}(t)
\end{equation}
for all $t \in [0,+\infty).$
\eqref{enought prove} is trivial if either $t \ge t_2$ or $2t < t_1.$ Whenever $t \in [t_1 /2, t_2],$ by the monotonicity of $\Psi_{p,\lambda}$ we have that
\begin{equation*}
\frac{\Psi_{p,\lambda} (2t)}{\Psi_{p,\lambda}(2t_2)} \le 1 \le
\frac{\Psi_{p,\lambda} (t)}{\Psi_{p,\lambda}(t_1 /2)}  .
\end{equation*}

We next prove (P-3) for $\Psi_{p,\lambda} .$
For any $t \in (0,t_1),$ we have that
\begin{equation}\label{jkl1}
\Psi'_{p, \lambda} (t) = pt^{p-1} = p \frac{\Psi_{p,\lambda} (t)}{t} .
\end{equation}
Since $p \log(e+t) +\lambda t(e+t)^{-1} \le p \log(e+t)$ for all $t \in [t_2, +\infty),$ from \eqref{Phi'} we have that 
\begin{equation}\label{jkl2}
\Psi'_{p,\lambda} (t) \le p t^{p-1} \log^{\lambda} (e+t) = p
\frac{\Psi_{p,\lambda} (t)}{t}  
\end{equation}
for all $t \in [t_2, +\infty) .$
Moreover, since $t^p _1 -t_1 k <0$ it follows from \eqref{Phi_M definition} that
\begin{equation}\label{jkl3}
\frac{\Psi_{p,\lambda} (t)}{t}
= k + \frac{t^p_1 -t_1 k}{t} \ge k + \frac{t^p_1 -t_1 k}{t_1} =t^{p-1} _1 = \frac{k}{t^{p-1} _1} \Psi'_{p,\lambda} (t)
\end{equation}
for all $t \in [t_1,t_2).$
By \eqref{jkl1}, \eqref{jkl2} and \eqref{jkl3}, we finish (P-3) for $\Psi_{p,\lambda} .$

We next prove (P-4). Let $s \in (0 , t_1 / t_2) .$ 
If $ t \in (0,t_1) ,$ from \eqref{Phi_M definition} we have that
\begin{equation}\label{M(st)/M(t):1}
\frac{ \Psi_{p,\lambda}(st)}{\Psi_{p,\lambda}(t)}= \frac{(st)^p}{t^p}= s^p \qquad \forall t \in (0,t_1).
\end{equation}
If $t \in [t_1,t_2) ,$ from \eqref{Phi_M definition} we have that
\begin{equation}\label{M(st)/M(t):2}
\frac{\Psi_{p,\lambda} (st)}{\Psi_{p,\lambda} (t)} = \frac{(st)^p}{k(t-t_1) +t^{p}_1}
\le \left(\frac{t_2 }{t_1} \right)^p s^p \qquad  \forall t \in [t_1,t_2).
\end{equation}
Whenever $t \in [t_2, t_1/s),$ it follows from \eqref{Phi_M definition} that
\begin{equation}\label{M(st)/M(t):3}
\frac{\Psi_{p,\lambda}(st)}{\Psi_{p,\lambda}(t)} = \frac{(st)^p}{\Phi_{p,\lambda}(t)}\le \frac{s^p}{\log^{\lambda}(e+t_1 s^{-1})}.
\end{equation}
If $t \in [t_1 /s, t_2 /s ),$ we obtain from \eqref{Phi_M definition} that 
\begin{equation}\label{M(st)/M(t):4}
\frac{\Psi_{p,\lambda}(st)}{\Psi_{p,\lambda}(t)} = \frac{k(st-t_1) +t^{p} _1}{\Phi_{p,\lambda}(t)}
\le \frac{k(t_2 -t_1) +t^{p} _1}{\Phi_{p,\lambda}(t_1 s^{-1})} = \frac{\Phi_{p,\lambda}(t_2)}{t^{p}_1} \frac{s^p}{\log^{\lambda}(e+t_1 s^{-1})} .
\end{equation}
For any $t \in [t_2 /s, +\infty),$ from \eqref{Phi_M definition} we have that
\begin{equation}\label{M(st)/M(t):5-1}
\frac{\Psi_{p,\lambda}(st)}{\Psi_{p,\lambda} (t)} = \frac{\Phi_{p,\lambda}(st)}{\Phi_{p,\lambda}(t)} = s^p \left(\frac{\log(e+st)}{\log(e+t)}\right)^{\lambda}.
\end{equation}
Moreover by the monotonicity of function $(\log(s) +\log(e+\cdot))\log^{-1}(e+\cdot),$ it follows that 
\begin{equation*}
\frac{\log(e+st)}{\log(e+t)} \ge \frac{\log(s) +\log(e+t)}{\log(e+t)}
 \ge \frac{\log(s) +\log(e+t_2 s^{-1})}{\log(e+t_2 s^{-1})} \ge \frac{\log(t_2)}{\log(e+t_2 s^{-1})}
\end{equation*}
for all $t \ge t_2 /s .$
Hence we derive from \eqref{M(st)/M(t):5-1} that
\begin{equation}\label{M(st)/M(t):5}
\frac{\Psi_{p,\lambda} (st)}{\Psi_{p,\lambda} (t)} \le \log^{\lambda}(t_2) \frac{s^p}{\log^{\lambda}(e+t_2 s^{-1})}
\end{equation}
for all $t \in [t_2 /s, +\infty).$
Combining \eqref{M(st)/M(t):1}, \eqref{M(st)/M(t):2}, \eqref{M(st)/M(t):3}, \eqref{M(st)/M(t):4} with \eqref{M(st)/M(t):5} implies that there is a constant $C>0$ such that
\begin{equation}\label{M(st)/M(t):rough}
\frac{\Psi_{p,\lambda} (st)}{\Psi_{p,\lambda}(t)} \le C \frac{s^p}{\log^{\lambda}(e+t_2 s^{-1})}
\end{equation}
for all $t \in (0,+\infty )$ and $0< s <  t_1 / t_2.$
For a given $r \in (0,p),$ since $C\frac{s^p}{\log^{\lambda}(e+t_2 s^{-1})} \le s^r$ for all $s \ll 1,$ it follows from \eqref{M(st)/M(t):rough} that there is $s_0 >0$ such that
\begin{equation}\label{M(st)/M(t):0 nbhd}
\Psi_{p,\lambda}(st) \le s^r \Psi_{p,\lambda}(t)
\end{equation}
for all $t \in [0,+\infty)$ and $s \in [0,s_0).$

Whenever $s \in [s_0,1],$ by the monotonicity of $\Psi_{p,\lambda}$ we have that
\begin{equation}\label{M(st)/M(t):1 nbhd}
\Psi_{p,\lambda} (st) \le \Psi_{p,\lambda} (t) \le s^{-r} _{0} s^{r} \Psi_{p,\lambda} (t)
\end{equation}
for all $t \in [0,+\infty)$ and $r>0.$
Combining \eqref{M(st)/M(t):0 nbhd} with \eqref{M(st)/M(t):1 nbhd} concludes that there is a constant $C>0$ such that
\begin{equation*}
\Psi_{p,\lambda}(st) \le C s^{r} \Psi_{p,\lambda}(t)
\end{equation*}
for all $t \in [0,+\infty)$ $s \in [0,1]$ and $r \in (0,p).$
\end{proof}

\begin{remark}
For $p>1$ and $\lambda <0,$ let $\Psi_{p,\lambda}$ be as in \eqref{Phi_M definition} and $\Phi_{p,\lambda}$ be as in \eqref{Phi definition}.
Under the same assumptions as in Lemma \ref{weighted muckenhoupt},
we have that
\begin{equation}\label{remark 1}
\int_{\mathbb{R}^2} \Psi_{p,\lambda}(M_f (x)) w(x) \, dx \lesssim \int_{\mathbb{R}^2}  \Psi_{p,\lambda}(|f(x)|) w(x) \, dx .
\end{equation}
Since $\lim_{t \rightarrow 0+} \Psi_{p,\lambda} (t) /\Phi_{p,\lambda}(t) =1,$ it follows that
\begin{equation}\label{remark 1-1}
\Psi_{p,\lambda} (t) \approx \Phi_{p,\lambda} (t) 
\end{equation}
whenever $t \in [0,+\infty).$
Hence we derive from \eqref{remark 1} that 
\begin{equation*}\label{remark 2}
\int_{\mathbb{R}^2} \Phi_{p,\lambda}(M_f (x)) w(x) \, dx \lesssim \int_{\mathbb{R}^2}  \Phi_{p,\lambda}(|f(x)|) w(x) \, dx .
\end{equation*}
\end{remark}

\section{Proof of Theorem \ref{main thm}}
We begin by proving the following special case of Theorem \ref{main thm}. 

\begin{theorem}\label{main theorem}
Let $\varphi: \mathbb{S}^1 \rightarrow \mathbb{S}^1$ be a homeomorphism, and $h = P[\varphi] : \mathbb{D} \rightarrow \mathbb{D}$ be the harmonic extension of $\varphi.$
For any $p >1,$ we have that
\begin{enumerate}
\item[(1)] if either $\alpha \in (p-2, +\infty)$ and $\lambda \in \mathbb{R},$ or $\alpha=p-2$ and $\lambda \in (-\infty, -1),$
\begin{equation*}
\mbox{both }I_1 (p,\alpha ,\lambda ,h) \mbox{ and } I_2 (p,\alpha ,\lambda ,h) \mbox{ are finite.}
\end{equation*}
\item[(2)] if either $\alpha \in (-1,p-2)$ and $\lambda \in \mathbb{R},$ or $\alpha=p-2$ and $\lambda \in [-1,+\infty),$ then
\begin{equation*}
\mbox{both }I_1 (p,\alpha ,\lambda ,h) \mbox{ and } I_2 (p,\alpha ,\lambda ,h) \mbox{ are comparable to } \mathcal{U} (p,\alpha,\lambda,\varphi).
\end{equation*}
Moreover whenever $p \in (1,2]$ 
\begin{equation*}
\mbox{both }I_1 (p,\alpha ,\lambda ,h) \mbox{ and }I_2 (p,\alpha ,\lambda ,h)\mbox{ dominate }\mathcal{V} (p,\alpha,\lambda,\varphi),
\end{equation*}
while 
\begin{equation*}
\mathcal{V} (p,\alpha,\lambda,\varphi)\mbox{ controls both }I_1 (p,\alpha ,\lambda ,h)\mbox{ and }I_2 (p,\alpha ,\lambda ,h)
\end{equation*}
for all $p \in [2,+\infty) .$ 
Furthermore both $I_1 (p,\alpha ,\lambda ,h)$ and $I_2 (p,\alpha ,\lambda ,h)$ are in general comparable to $\mathcal{V} (p,\alpha,\lambda,\varphi)$ only for $p =2 .$
\item[(3)] if either $\alpha \in (-\infty, -1)$ and $\lambda \in \mathbb{R},$ or $\alpha=-1$ and $\lambda \in [-1,+\infty),$ we have that $I_1 (p,\alpha ,\lambda ,h)=\infty$. While $I_2 (p,\alpha ,\lambda ,h)=\infty$ for all $\alpha \in (-\infty, -1]$ and $\lambda \in \mathbb{R}.$
\end{enumerate}
\end{theorem}

Let $\varphi : \mbs^1 \rightarrow \mbs^1 $ be a homeomorphism.
Given $p >1, \alpha \in \mathbb{R}, \lambda \in \mathbb{R} ,$ 
we define
\begin{equation*}\label{mathcal E _1}
\mathcal{E}_1 (p,\alpha,\lambda,\varphi)= \sum_{j=1} ^{+\infty} \sum_{k=1} ^{2^j} \ell (\varphi(\Gamma_{j,k}))^{p} \ell(\Gamma_{j,k})^{2+\alpha -p} j^{\lambda}
\end{equation*}
and 
\begin{equation*}\label{mathcal E _2}
\mathcal{E}_2 (p,\alpha,\lambda,\varphi) = \sum_{j=1} ^{+\infty} \sum_{k=1} ^{2^j} \Phi_{p ,\lambda} \big( \frac{\ell(\varphi(\Gamma_{j,k}))}{\ell(\Gamma_{j,k})}\big)
\ell (\Gamma_{j,k})^{2+\alpha} 
\end{equation*}
where $\Phi_{p,\lambda}(t)$ is from \eqref{Phi definition}.

\begin{lemma}\label{connect mathcal E_1 and mathcal E_2}
Let $\varphi: \mathbb{S}^1 \rightarrow \mathbb{S}^1$ be a homeomorphism. For any $p >1, \alpha \in (-1,+\infty)$ and every $\lambda \in \mathbb{R},$ the dyadic energy $\mathcal{E}_1 (p,\alpha,\lambda,\varphi)$ and $\mathcal{E}_2 (p,\alpha,\lambda,\varphi)$ are equivalent.
\end{lemma}

\begin{proof}
We first consider the case $\lambda \ge 0$.
Let $\Phi_{p,\lambda}$ be as in \eqref{Phi definition}.
Since $\ell(\varphi(\Gamma_{j,k})) \le 2 \pi$ and $\ell(\Gamma_{j,k}) \approx 2^{-j}$ for all $j \in \mathbb{N}$ and $k \in \{1,..., 2^j\},$ by the monotonicity and $\Delta_2$-property of the standard logarithm we have that
\begin{equation*}
\Phi_{p,\lambda} \big( \frac{\ell(\varphi(\Gamma_{j,k}))}{\ell(\Gamma_{j,k})}\big)
\lesssim \big( \frac{\ell(\varphi(\Gamma_{j,k}))}{\ell(\Gamma_{j,k})}\big)^p \log^{\lambda}\big(e+  2\pi \cdot 2^j \big)
\lesssim \big( \frac{\ell(\varphi(\Gamma_{j,k}))}{\ell(\Gamma_{j,k})}\big)^p  j^{\lambda} .
\end{equation*}
Hence
\begin{equation}\label{mathcal E: 1-1}
\mathcal{E}_2 (p,\alpha,\lambda,\varphi) \lesssim \mathcal{E}_1 (p,\alpha,\lambda,\varphi).
\end{equation}
Given $p >1$ and $\alpha \in (-1,+\infty),$ there is $\beta \in (0,1)$ such that $\alpha >(1-\beta) p -1>-1.$
Define 
\begin{equation*}
\chi_{j,k}=
\begin{cases}
1& \mbox{if } \ell(\varphi(\Gamma_{j,k})) \ge 2^{-j \beta}, \\
0 & \mbox{otherwise}.  
\end{cases}
\end{equation*}
We decompose $\mathcal{E}_1 (p,\alpha,\lambda,\varphi)$ as
\begin{align}\label{mathcal E: 1}
\mathcal{E}_1 (p,\alpha,\lambda,\varphi) =& 
\sum_{j=1} ^{+\infty} \sum_{k=1} ^{2^j} \ell (\varphi(\Gamma_{j,k}))^{p} \ell(\Gamma_{j,k})^{2+\alpha -p} j^{\lambda} \chi_{j,k} \notag \\
& + \sum_{j=1} ^{+\infty} \sum_{k=1} ^{2^j} \ell (\varphi(\Gamma_{j,k}))^{p} \ell(\Gamma_{j,k})^{2+\alpha -p} j^{\lambda} (1-\chi_{j,k}) \notag \\
=: & \mathcal{E}'_1 (p,\alpha,\lambda,\varphi)  + \mathcal{E}''_1 (p,\alpha,\lambda,\varphi) .
\end{align}
Whenever $\ell(\varphi(\Gamma_{j,k})) \ge 2^{-j \beta} ,$
by \eqref{dyadic deom 1-0} we have 
$j^{\lambda} \lesssim \log^{\lambda} \left(e+ \ell(\varphi(\Gamma_{j,k})) \ell(\Gamma_{j,k})^{-1} \right) .$  Therefore 
\begin{align}\label{mathcal E: 2}
\mathcal{E}' _1 (p,\alpha,\lambda,\varphi) 
\lesssim &
\sum_{j=1} ^{+\infty} \sum_{k=1} ^{2^j} \ell(\varphi(\Gamma_{j,k}))^{p} \ell(\Gamma_{j,k})^{2+\alpha -p} \log^{\lambda} \left(e+ \frac{\ell(\varphi(\Gamma_{j,k}))}{\ell(\Gamma_{j,k})} \right) \notag\\
= & \mathcal{E}_2 (p,\alpha,\lambda,\varphi).
\end{align}
Moreover, by \eqref{dyadic deom 1-0} we have that
\begin{equation}\label{mathcal E: 3}
\mathcal{E}'' _1 (p,\alpha,\lambda,\varphi) 
\le
\sum_{j=1} ^{+\infty} \sum_{k=1} ^{2^j}  2^{-\beta j  p} 2^{j(p-2-\alpha)} j^{\lambda}
=  \sum_{j=1} ^{+\infty} 2^{-j((\beta-1)p +1+\alpha)} j^{\lambda} <\infty.
\end{equation}
We conclude from \eqref{mathcal E: 1}, \eqref{mathcal E: 2} and \eqref{mathcal E: 3} that there is a constant $C>0$ such that 
\begin{equation}\label{mathcal E: 3-1}
\mathcal{E}_1 (p,\alpha,\lambda,\varphi) \lesssim C+ \mathcal{E}_2 (p,\alpha,\lambda,\varphi) .
\end{equation}
From \eqref{mathcal E: 1-1} and \eqref{mathcal E: 3-1} 
it follows that 
\begin{equation}\label{mathcal E: 3-2}
\mathcal{E}_1 (p,\alpha,\lambda,\varphi) \mbox{ and } \mathcal{E}_2 (p,\alpha,\lambda,\varphi) \mbox{ are comparable whenever } \lambda \ge 0.
\end{equation}
Analogously to \eqref{mathcal E: 3-2}, we have that
$\mathcal{E}_1 (p,\alpha,\lambda,\varphi)$ and $\mathcal{E}_2 (p,\alpha,\lambda,\varphi)$ are comparable whenever $\lambda < 0.$
\end{proof}

\begin{lemma}\label{I_1 discrete}
Let $\varphi :\mathbb{S}^1 \rightarrow \mathbb{S}^1$ be a homeomorphism, and $h=P[\varphi] : \mathbb{D} \rightarrow \mathbb{D}$ be the Poisson homeomorphic extension of $\varphi.$
For any $p>1$, we have that $I_1 (p,\alpha ,\lambda ,h) \lesssim \mathcal{E}_1 (p,\alpha,\lambda,\varphi)$ whenever $\alpha \in (-1,+\infty)$ and $\lambda \in \mathbb{R},$ while $\mathcal{E}_1 (p,\alpha,\lambda,\varphi)$ is controlled by $ I_1 (p,\alpha ,\lambda ,h)$ for all $\alpha \in (-1,p-1)$ and $\lambda \in \mathbb{R}.$
\end{lemma}

\begin{proof}
We first prove that $I_1 (p,\alpha ,\lambda ,h) \lesssim \mathcal{E}_1 (p,\alpha,\lambda,\varphi) $ for all $\alpha >-1$ and all $\lambda \in \mathbb{R} .$   
Let $w_{\alpha ,\lambda}$ be as in \eqref{A_p weight}.
For any $j \in \mathbb{N}$ and $1 \le k \le 2^j,$ by \eqref{dyadic deom 1-1} and \eqref{dyadic deom 1-0} we have that
\begin{equation}\label{int_D decom 0}
w_{\alpha ,\lambda} (z) \approx 2^{-j \alpha} j^{\lambda} \approx  \ell(\Gamma_{j,k})^{\alpha} j^{\lambda} 
\end{equation}
for all $z \in Q_{j,k}.$
Hence
\begin{equation}\label{int_D decom}
I_1 (p,\alpha ,\lambda ,h)
\approx  \sum_{j=1}^{+\infty} \sum_{k=1} ^{2^j}  2^{-j \alpha} j^{\lambda} \int_{Q_{j,k}} |D h (z)|^p \, dz.
\end{equation}
Let $\mathcal{P}(\Gamma_{j,k})$ be the technical decomposition of $\mbs^1$ based on $\Gamma_{j,k}$ in \cite[Section 2.1]{our paper}. 
As shown in \cite[Proof (iv) $\Rightarrow$ (i)]{our paper}, for any $j \in \mathbb{N}$ and $k =1,...,2^j$ we have that 
\begin{equation}\label{h_z (z) <}
|D h (z)|
\lesssim  \sum_{n \le j } \sum_{m \in i_n} \frac{\ell(\varphi(\Gamma_{n,m}))}{2^{-n}}.
\end{equation}
for all $z \in Q_{j,k} .$ 
Here $\Gamma_{n,m} \in \mathcal{P}(\Gamma_{j,k})$ and $\sharp i_n \le 3$ for all $n \le j ,$ see \cite[Section 2.1]{our paper}.
Let $\alpha>-1 .$ There is $q_0 >1$ such that $p/q_0 -1-\alpha <0.$
Denote by $p_0$ the exponent conjugate to $q_0.$
Via H\"{o}lder's inequality we derive from \eqref{h_z (z) <} that
\begin{align}\label{|h_z (z)|^p <}
|Dh (z)|^p \lesssim &
\left(\sum_{n \le j} \sum_{m \in i_n} \frac{\ell(\varphi(\Gamma_{n,m}))}{2^{-n(\frac{1}{q_0}+ \frac{1}{p_0} )}}\right)^p
\le  \left(\sum_{n \le j} \sum_{m \in i_n} 2^{\frac{nq}{q_0}}\right)^{\frac{p}{q}} \sum_{n \le j} \sum_{m \in i_n} \frac{\ell(\varphi(\Gamma_{n,m}))^{p}}{2^{-\frac{np}{p_0}}}  \notag \\
\approx & 2^{\frac{jp}{q_0}} \sum_{n \le j} \sum_{m \in i_n} \frac{\ell(\varphi(\Gamma_{n,m}))^{p}}{2^{-\frac{np}{p_0}}}
\end{align}
for all $z \in Q_{j,k}.$
By \eqref{dyadic deom 1-2}, \eqref{int_D decom} and \eqref{|h_z (z)|^p <} we have that
\begin{equation}\label{|h_z (z)|^p <1}
I_1 (p,\alpha ,\lambda ,h) \lesssim \sum_{j=1} ^{+\infty} \sum_{k=1} ^{2^j} 2^{j(\frac{p}{q_0} -2-\alpha)} j^{\lambda} \sum_{n \le j} \sum_{m \in i_n} \frac{\ell(\varphi(\Gamma_{n,m}))^{p}}{2^{-\frac{np}{q_0}}}.
\end{equation}
Moreover given a dyadic arc $\Gamma_{n,m},$ for any $j \geq n$ it is shown in \cite[Section 2.1]{our paper} that  
\begin{equation}\label{number}
\sharp \{\Gamma: \Gamma \mbox{ is a }j\mbox{-level dyadic arc and }\Gamma_{n,m} \in \mathcal{P}(\Gamma)  \}\leq 3 \cdot 2^{j-n} .
\end{equation}
From Fubini's theorem and \eqref{number} we obtain that
\begin{align}\label{|h_z (z)|^p <1 :1}
&\sum_{j=1} ^{+\infty} \sum_{k=1} ^{2^j} 2^{j(\frac{p}{q_0} -2-\alpha)} j^{\lambda} \sum_{n \le j} \sum_{m \in i_n} \frac{\ell(\varphi(\Gamma_{n,m}))^{p}}{2^{-\frac{np}{p_0}}} \notag \\
= & \sum_{n=1} ^{\infty} \sum_{m=1} ^{2^n} \frac{\ell(\varphi(\Gamma_{n,m}))^{p}}{2^{-\frac{np}{p_0}}} \sum_{n \le j} \sum_{k} 2^{j(\frac{p}{q_0} -2-\alpha)} j^{\lambda} \notag\\
\lesssim & \sum_{n=1} ^{\infty} \sum_{m=1} ^{2^n} \frac{\ell(\varphi(\Gamma_{n,m}))^{p}}{2^{-\frac{np}{p_0}}} \sum_{ n \le j} 2^{j(\frac{p}{q_0} -2-\alpha)} j^{\lambda} 2^{j-n} \notag\\
= & \sum_{n=1} ^{\infty} \sum_{m=1} ^{2^n} \ell(\varphi(\Gamma_{n,m}))^{p} 2^{n(\frac{p}{p_0} -1)} \sum_{n \le j} 2^{j(\frac{p}{q_0} -1-\alpha)} j^{\lambda} .
\end{align}
Moreover when $p/q_0 -1-\alpha <0$ we have that
\begin{equation}\label{|h_z (z)|^p <1 :1-1}
\sum_{n \le j} 2^{j(\frac{p}{q_0} -1-\alpha)} j^{\lambda} \approx 2^{n(\frac{p}{q_0} -1-\alpha)} n^{\lambda} .
\end{equation}
By \eqref{|h_z (z)|^p <1}, \eqref{|h_z (z)|^p <1 :1}, \eqref{|h_z (z)|^p <1 :1-1} and \eqref{dyadic deom 1-0}, we conclude that
\begin{equation*}
I_1 (p,\alpha ,\lambda ,h) \lesssim 
\sum_{n=1} ^{\infty} \sum_{m=1} ^{2^n} \ell(\varphi(\Gamma_{n,m}))^{p} 2^{n(p-2-\alpha)} n^{\lambda} 
\approx  \mathcal{E}_1 (p,\alpha,\lambda,\varphi).
\end{equation*}

We next prove that $\mathcal{E}_1 (p,\alpha,\lambda,\varphi)$ is controlled by $I_1 (p,\alpha ,\lambda ,h) $ for all $\alpha \in (-1,p-1)$ and $\lambda \in \mathbb{R}.$
By \cite[($3.17$)]{our paper}, there is $j_0 >1$ such that
\begin{equation}\label{|h_z (z)|^p <1 :2}
\ell(\varphi(\Gamma_{j,k})) \lesssim \frac{1}{\ell(\Gamma_{j,k})} \int_{CQ_{j,k} \cap \mathbb{D}} |Dh(z)| \, dz
\end{equation}
for all $j \ge j_0$ and $k \in \{1,..., 2^j\}.$
Set $H(z) = |Dh(z)| \chi_{\mathbb{D}} (z).$ 
By \eqref{dyadic deom 2} we have that
\begin{equation}\label{|h_z (z)|^p <1 :2-1}
\displaystyle  \mint{-}_{CQ_{j,k} \cap \mathbb{D}} |Dh(z)| dz
\le \displaystyle  \mint{-}_{CC'B_{j,k}} H(z) dz
\le \displaystyle  \mint{-}_{Q_{j,k}} M_{H}(w) dw,
\end{equation}
where the last inequality comes from the fact that $ \displaystyle \mint{-}_{CC'B_{j,k}} H(z) dz \le M_{H}(w)$ for all $w \in Q_{j,k}.$ Combining \eqref{|h_z (z)|^p <1 :2} with \eqref{|h_z (z)|^p <1 :2-1} implies that
\begin{equation} \label{ell varphi Gamma_j,k <}
\ell(\varphi(\Gamma_{j,k})) \lesssim \ell(\Gamma_{j,k}) \displaystyle \mint{-}_{Q_{j,k}} M_{H}(z) dz 
\end{equation}
for all $j \ge j_0$ and $k =1,...,2^j .$
By Jensen's inequality and \eqref{dyadic deom 1-2}, we deduce from \eqref{ell varphi Gamma_j,k <} that
\begin{equation}\label{sum_ j>j_0 -1}
\ell (\varphi(\Gamma_{j,k}))^{p} \lesssim \ell (\Gamma_{j,k})^{p-2} \int_{Q_{j,k}} M^{p} _{H} (z) \, dz .
\end{equation}
By \eqref{int_D decom 0} and \eqref{sum_ j>j_0 -1}, there is then a constant $C>0$ such that
\begin{equation}\label{sum_ j>j_0}
\mathcal{E}_1 (p,\alpha,\lambda,\varphi) 
\lesssim C+ \sum_{j=j_0} ^{+\infty} \sum_{k =1} ^{2^j} \int_{Q_{j,k}} M^{p} _{H}(z) w_{\alpha ,\lambda}(z) \, dz 
\le C+ \int_{\mathbb{R}^2} M^{p} _{H} (z) w_{\alpha ,\lambda}(z) \, dz .
\end{equation}
Moreover, for any $\alpha \in (-1,p-1)$ and $\lambda \in \mathbb{R},$
from Proposition \ref{proposition: A_p weight} and Lemma \ref{weighted muckenhoupt} it follows that
\begin{equation}\label{sum_ j>j_0 :1}
\int_{\mathbb{R}^2} M^{p} _{H} (z)  w_{\alpha ,\lambda}(z) \, dz
\lesssim  \int_{\mathbb{R}^2} H^p (z)  w_{\alpha ,\lambda}(z) \, dz =I_1 (p,\alpha ,\lambda ,h).
\end{equation}
By \eqref{sum_ j>j_0} and \eqref{sum_ j>j_0 :1} we conclude that $\mathcal{E}_1 (p,\alpha,\lambda,\varphi)$ is controlled by $ I_1 (p,\alpha ,\lambda ,h) $ for all $\alpha \in (-1,p-1)$ and $\lambda \in \mathbb{R} .$
\end{proof}

\begin{lemma}\label{desvribe I_2} 
Let $\varphi :\mathbb{S}^1 \rightarrow \mathbb{S}^1$ be a homeomorphism, and $h=P[\varphi] : \mathbb{D} \rightarrow \mathbb{D}$ be the Poisson homeomorphic extension of $\varphi .$ For any $p>1$, we have that $I_2 (p,\alpha ,\lambda ,h) \lesssim \mathcal{E}_1 (p,\alpha,\lambda,\varphi)$ whenever $\alpha \in (-1,+\infty)$ and $\lambda \in \mathbb{R},$ while $\mathcal{E}_2 (p,\alpha,\lambda,\varphi)$ is controlled by $I_2 (p,\alpha ,\lambda ,h)$ for all $\alpha \in (-1,p-1)$ and $\lambda \in \mathbb{R}.$
\end{lemma}

\begin{proof}
We first consider that case $\lambda \ge 0 .$
Let $\Phi_{p,\lambda}$ be as in \eqref{Phi definition}.
As (P-1) and (P-2) in Proposition \ref{proposition of Phi: lamdba >0} show that $\Phi_{p ,\lambda} (t)$ is increasing and satisfies $\Delta_2$-property on $[0,+\infty) .$  
From \eqref{int_D decom 0} and \eqref{h_z (z) <}
we have that
\begin{align}\label{int_D Phi(|h_z (z)|) delta(z)^ alpha dz le: 1}
I_2 (p,\alpha ,\lambda ,h)
=  & \sum_{j=1} ^{+\infty} \sum_{k=1} ^{2^j} \int_{Q_{j,k}} \Phi_{p, \lambda}(|Dh (z)|) w_{\alpha ,0} (z) \, dz \notag \\
\lesssim &  \sum_{j=1} ^{+\infty} \sum_{k=1} ^{2^j} \ell(\Gamma_{j,k})^{\alpha+2}  \Phi_{p,\lambda}\big(\sum_{n \le j} \sum_{m \in i_n} \frac{\ell(\varphi(\Gamma_{n,m}))}{2^{-n}} \big).
\end{align}
Moreover since $\ell({\varphi( \Gamma_{n, m})})\leq 2\pi$ for all $n \in \mathbb{N}$ and $m =1,...,2^n$, it follows that
\begin{equation*}\label{h_z (z) < :1}
\sum_{n\leq j}\sum_{m \in i_n} \frac{\ell({\varphi( \Gamma_{n, m})})}{2^{-n}}\lesssim  \sum_{n\le j} \frac{1}{2^{-n}}\lesssim 2^j.
\end{equation*}
for any $j \ge 1 .$
Therefore 
\begin{equation}\label{int_D Phi(|h_z (z)|) delta(z)^ alpha dz le: 1-1}
\log^{\lambda} \big( e+\sum_{n \le j} \sum_{m} \frac{\ell(\varphi(\Gamma_{n,m}))}{2^{-n}} \big)
\lesssim \log^{\lambda} (e+ 2^j)
\lesssim j^{\lambda} 
\end{equation}
for all $j \ge 1.$
By \eqref{int_D Phi(|h_z (z)|) delta(z)^ alpha dz le: 1} and \eqref{int_D Phi(|h_z (z)|) delta(z)^ alpha dz le: 1-1} we obtain that
\begin{equation}\label{noname 1}
I_2 (p,\alpha ,\lambda ,h)
\lesssim \sum_{j=1} ^{+\infty} \sum_{k=1} ^{2^j} \ell(\Gamma_{j,k})^{\alpha+2} j^{\lambda} \big(\sum_{n \le j} \sum_{m \in i_n} \frac{\ell(\varphi(\Gamma_{n,m}))}{2^{-n}} \big)^p.
\end{equation}
The analogous arguments as for $I_1 (p,\alpha ,\lambda ,h) \lesssim \mathcal{E}_1 (p,\alpha,\lambda,\varphi)$ in Lemma \ref{I_1 discrete} imply that
\begin{equation}\label{noname 2}
\sum_{j=1} ^{+\infty} \sum_{k=1} ^{2^j} \ell(\Gamma_{j,k})^{\alpha+2} j^{\lambda} \big(\sum_{n \le j} \sum_{m \in i_n} \frac{\ell(\varphi(\Gamma_{n,m}))}{2^{-n}} \big)^p
\lesssim  \mathcal{E}_1 (p,\alpha,\lambda,\varphi).
\end{equation}
We conclude from \eqref{noname 1} and \eqref{noname 2} that $I_2 (p,\alpha ,\lambda ,h)
\lesssim \mathcal{E}_1 (p,\alpha,\lambda,\varphi) .$

Applying $\Phi_{p,\lambda}$ to the both sides of \eqref{ell varphi Gamma_j,k <}, via (P-1) and (P-2) in Proposition \ref{proposition of Phi: lamdba >0} and Jensen's inequality we have that
\begin{equation}\label{Phi ell varphi Gamma_j,k /ell Gamma_j,k}
\Phi_{p,\lambda} \big(\frac{\ell(\varphi(\Gamma_{j,k}))}{\ell(\Gamma_{j,k})} \big) 
\lesssim \Phi_{p,\lambda} \big( \displaystyle \mint{-}_{Q_{j,k}} M_{H}(z) dz \big) 
\le \displaystyle \mint{-}_{Q_{j,k}} \Phi_{p,\lambda}(M_{H} (z)) dz
\end{equation}
for all $j \ge j_0$ and $k \in \{1,...,2^j\}.$
By \eqref{dyadic deom 1-2}, \eqref{int_D decom 0} and \eqref{Phi ell varphi Gamma_j,k /ell Gamma_j,k}, we then obtain that
\begin{align}\label{no title 1}
\mathcal{E}_2 (p,\alpha,\lambda,\varphi) 
\lesssim & \sum_{j=1} ^{+\infty} \sum_{k=1} ^{2^j} \int_{Q_{j,k}}\Phi_{p,\lambda}(M_{H} (z)) w_{\alpha ,0} (z) \, dz \notag \\
\le & \int_{\mathbb{R}^2} \Phi_{p,\lambda}(M_{H} (z)) w_{\alpha ,0} (z) \, dz .
\end{align}
Moreover, for any $\alpha \in (-1,p-1)$ it follows from Lemma \ref{weighted muckenhoupt} that
\begin{equation}\label{no title 2}
\int_{\mathbb{R}^2} \Phi_{p,\lambda}(M_{H} (z)) w_{\alpha ,0} (z) \, dz
\lesssim \int_{\mathbb{R}^2} \Phi_{p,\lambda}(H (z)) w_{\alpha ,0} (z) \, dz
= I_2 (p,\alpha ,\lambda ,h).
\end{equation}
By \eqref{no title 1} and \eqref{no title 2} we conclude that $\mathcal{E}_2 (p,\alpha,\lambda,\varphi) \lesssim I_2 (p,\alpha ,\lambda ,h) .$

We next consider the case $\lambda<0.$
Let $\Psi_{p, \lambda}$ be as in \eqref{Phi_M definition}.
By the analogous arguments as for \eqref{int_D Phi(|h_z (z)|) delta(z)^ alpha dz le: 1}, we have that
\begin{equation}\label{lambda <0 : 1}
\int_{\mathbb{D}} \Psi_{p,\lambda}(|D h (z)|) w_{\alpha ,0} (z) \, dz 
\lesssim \sum_{j=1}^{+\infty} \sum_{k=1}^{2^j} \ell(\Gamma_{j,k})^{\alpha+2} \Psi _{p,\lambda} \big( \sum_{n\le j} \sum_{m \in i_n} \frac{\ell(\varphi(\Gamma_{n,m}))}{2^{-n}}\big).
\end{equation}
Set $S_{j,k}= \sum_{n \le j} \sum_{m \in i_n} \frac{\ell(\varphi(\Gamma_{n,m}))}{2^{-n}}.$
It follows from \eqref{remark 1-1} and \eqref{lambda <0 : 1} that
\begin{equation}\label{lambda <0 : 4}
I_2 (p,\alpha ,\lambda ,h) 
\lesssim \sum_{j=1}^{+\infty} \sum_{k=1}^{2^j} \ell(\Gamma_{j,k})^{\alpha+2} \Phi_{p,\lambda} ( S_{j,k} ) .
\end{equation}
Since $\alpha >-1,$ there is $\beta >0$ such that $\beta p \le 1+\alpha.$
Define
\begin{equation*}
\chi (j,k) = 
\begin{cases}
1 & \mbox{if } S_{j,k} < 2^{j \beta}, \\
0 & \mbox{otherwise}.
\end{cases}
\end{equation*}
Then 
\begin{align}\label{lambda <0 : 4-1}
\sum_{j=1}^{+\infty} \sum_{k=1}^{2^j} \ell(\Gamma_{j,k})^{\alpha+2} \Phi_{p,\lambda} ( S_{j,k} )
= & \sum_{j=1}^{+\infty} \sum_{k=1}^{2^j} \ell(\Gamma_{j,k})^{\alpha+2} \Phi_{p,\lambda} (\chi (j,k) S_{j,k}) \notag \\
& + \sum_{j=1}^{+\infty} \sum_{k=1}^{2^j} \ell(\Gamma_{j,k})^{\alpha+2} \Phi_{p,\lambda} ((1-\chi (j,k))   S_{j,k})=: \sum\nolimits _{1} + \sum\nolimits _{2} .
\end{align}
Since $\log^{\lambda} (e+S_{j,k}) \le \log^{\lambda} (e) = 1,$
we have that
\begin{equation}\label{lambda <0 : 5}
\sum\nolimits _{1}
\le \sum\nolimits _{1} 2^{-j(\alpha+2)} (S_{j,k})^p
\le \sum_{j=1} ^{\infty} 2^{j(p \beta  -\alpha -1)} < \infty.
\end{equation}
Whenever $S_{j,k} \ge  2^{j \beta} ,$ it follows that $\log^{\lambda} (e +S_{j,k}) \lesssim j^{\lambda}.$
Via the analogous arguments as for $I_1 (p,\alpha ,\lambda ,h) \lesssim \mathcal{E}_1 (p,\alpha,\lambda,\varphi)$ in Lemma \ref{I_1 discrete}, it then follows that 
\begin{equation}\label{lambda <0 : 6}
\sum\nolimits_{2}
\le \sum_{j=1} ^{\infty} \sum_{k=1}^{2^j}  \ell(\Gamma_{j,k})^{\alpha+2} j^{\lambda} S_{j,k} ^p 
\lesssim \mathcal{E}_1 (p,\alpha,\lambda,\varphi).
\end{equation}
From \eqref{lambda <0 : 4}, \eqref{lambda <0 : 4-1}, \eqref{lambda <0 : 5} and \eqref{lambda <0 : 6}, we conclude that there is a constant $C>0$ such that $I_2 (p,\alpha ,\lambda ,h) \lesssim C+ \mathcal{E}_1 (p,\alpha,\lambda,\varphi) .$

By the analogous arguments as for $\mathcal{E}_2 (p,\alpha,\lambda,\varphi) \lesssim I_2 (p,\alpha ,\lambda ,h)$ whenever $\lambda \ge 0,$ we have that
\begin{equation}\label{lambda <0 : 7}
\sum_{j=1} ^{+\infty} \sum_{k=1} ^{2^j} \ell(\Gamma_{j,k})^{\alpha+2} \Psi_{p,\lambda} \big(\frac{\ell(\varphi(\Gamma_{j,k}))}{\ell(\Gamma_{j,k})} \big)
\lesssim \int_{\mathbb{R}^2} \Psi_{p,\lambda}(|Dh| (z)) w_{\alpha ,0}(z) \, dz .
\end{equation}
It follows from \eqref{remark 1-1} that $\mathcal{E}_2 (p,\alpha,\lambda,\varphi) \lesssim I_2 (p,\alpha ,\lambda ,h) .$
\end{proof}

\begin{proof}[\bf{Proof of Theorem \ref{main theorem} (1)}] 
From Lemma \ref{I_1 discrete} and Lemma \ref{desvribe I_2}, we have 
that both $I_1 (p,\alpha ,\lambda ,h)$ and $I_2 (p,\alpha ,\lambda ,h)$ are dominated by $\mathcal{E}_1 (p,\alpha,\lambda,\varphi)$ for all $p >1, \alpha \in (-1,+\infty)$ and each $\lambda \in \mathbb{R}.$
Moreover since $\ell(\varphi(\Gamma_{j,k})) \le 2 \pi$ for all $j \ge 1$ and $1 \le k \le 2^j,$ we have that
\begin{equation*}
\sum_{k =1} ^{2^j} \ell (\varphi(\Gamma_{j,k})) ^{p} \le (2 \pi )^{p-1} \sum_{k =1} ^{2^j} \ell (\varphi(\Gamma_{j,k})) = (2 \pi)^p.
\end{equation*}
Therefore both $I_1 (p,\alpha ,\lambda ,h)$ and $I_2 (p,\alpha ,\lambda ,h)$ are controlled by $\sum_{j=1} ^{\infty}  2^{j(p-2-\alpha)} j^{\lambda}$
whenever $\alpha \in (-1,+\infty)$ and $\lambda \in \mathbb{R}.$
Notice that $\sum_{j=1} ^{\infty}  2^{j(p-2-\alpha)} j^{\lambda} <\infty$ whenever either $p-2 < \alpha$ and $\lambda \in \mathbb{R},$ or $p-2 = \alpha$ and $\lambda < -1 .$
We hence complete Theorem \ref{main theorem} ($1$).
\end{proof}

By Example \ref{reasonable example}, there are homeomorphisms $\varphi : \mbs^1 \rightarrow \mbs^1$ such that, for their harmonic extensions $P[\varphi]$, both $I_1 (p,\alpha ,\lambda ,P[\varphi])$ and $I_2 (p,\alpha ,\lambda ,P[\varphi])$ may be finite or infinite for either some $\alpha \in (-1,p-2)$ and $\lambda \in \mathbb{R}$ or some $\alpha=p-2$ and $\lambda \in [-1,+\infty).$ 
How can we characterize both $I_1 (p,\alpha ,\lambda ,P[\varphi])<\infty$ and $I_2 (p,\alpha ,\lambda ,P[\varphi])< \infty$?
As shown in \cite{our paper}, double integrals of the inverse mapping over the boundary are potential choices.

\begin{lemma}\label{lemma: partial characterization}
Let $\varphi: \mathbb{S}^1 \rightarrow \mathbb{S}^1$ be a homeomorphism. 
For any $\alpha \in \mathbb{R}$ and $\lambda \in \mathbb{R},$
$\mathcal{V} (p,\alpha,\lambda,\varphi)$ is dominated by $\mathcal{E}_1 (p,\alpha,\lambda,\varphi)$ whenever $p \in (1,2]$; while $\mathcal{E}_1 (p,\alpha,\lambda,\varphi)$ is controlled by $\mathcal{V} (p,\alpha,\lambda,\varphi)$ if $p \in [2,+\infty)$.
\end{lemma}

\begin{proof}
We first consider the case $p \in (1,2] .$
Given $\xi \in \mathbb{S}^1$ and $t \ge 0,$ set
\begin{equation*}
E_t (\xi) =\{\eta \in \mbs^1 : |\varphi^{-1} (\xi) -\varphi^{-1} (\eta)|<t\}.
\end{equation*}
By Fubini's theorem we have that
\begin{align}\label{double integral: 2}
\int_{\mathbb{S}^1} \mathcal{A}_{p,\alpha ,\lambda} (|\varphi^{-1}(\xi)- \varphi^{-1}(\eta)|) \, |d \eta|
= &  \int_{\mathbb{S}^1} \int_{|\varphi^{-1}(\xi)- \varphi^{-1}(\eta)|} ^{1} -\mathcal{A}' _{p,\alpha ,\lambda}(t) \, dt \, |d \eta|   \notag \\
= &  \int_{0} ^{1} \int_{\mathbb{S}^1}  -\mathcal{A}' _{p,\alpha ,\lambda}(t) \chi_{E_t (\xi)} \, |d \eta| \, dt  \notag \\
=&  \int_{0} ^{1} - \mathcal{A}' _{p,\alpha ,\lambda}(t) \mathcal{L}^1 (E_t (\xi)) \, d t .
\end{align}
Moreover, from Jensen's inequality and Minkowski's inequality it follows that
\begin{align}\label{double integral: 3}
&\left(\int_{\mathbb{S}^1} \big( \int_{0} ^{1} - \mathcal{A}' _{p,\alpha ,\lambda}(t) \mathcal{L}^1 (E_t (\xi))\, d t  \big)^{p-1} \, |d \xi| \right)^{\frac{1}{p-1}} \notag \\
\lesssim & \left( \int_{\mathbb{S}^1} \big( \int_{0} ^{1} - \mathcal{A}' _{p,\alpha ,\lambda}(t) \mathcal{L}^1 (E_t (\xi)) \, d t  \big)^{\frac{1}{p-1}} \, |d\xi| \right)^{p-1} \notag \\
\le & \int_{0} ^{1}  \left(\int_{\mathbb{S}^1}  \left(- \mathcal{A}' _{p,\alpha ,\lambda}(t) \mathcal{L}^1 (E_t (\xi)) \right)^{\frac{1}{p-1}} \, |d \xi|  \right)^{ p-1} \, dt  \notag \\
= &  \int_{0} ^{1}  - \mathcal{A}' _{p,\alpha ,\lambda}(t) \left(\int_{\mathbb{S}^1}   \mathcal{L}^1 (E_t (\xi)) ^{\frac{1}{p-1}} \, |d \xi|  \right)^{ p-1} \, dt .
\end{align}
Combining \eqref{double integral: 2} with \eqref{double integral: 3} implies that
\begin{align}\label{double integral: 3-1}
\mathcal{V}^{\frac{1}{p-1}} (p,\alpha,\lambda,\varphi)
\lesssim & \int_{0} ^{1}  - \mathcal{A}' _{p,\alpha ,\lambda}(t) \left(\int_{\mathbb{S}^1}   \mathcal{L}^1 (E_t (\xi)) ^{\frac{1}{p-1}} \, |d \xi|  \right)^{ p-1 } \, dt \notag \\
\le &  \sum_{j=1} ^{+\infty} \int_{2^{-j}} ^{2^{1-j}}  - \mathcal{A}' _{p,\alpha ,\lambda}(t)\, dt \left(\int_{\mathbb{S}^1}   \mathcal{L}^1 (E_{2^{1-j}} (\xi)) ^{\frac{1}{p-1}} \, |d \xi|  \right)^{ p-1 } .
\end{align}
Since $E_{2^{1-j}} (\xi) \subset \cup_{i= k-1} ^{k+1} \varphi(\Gamma_{j,i}) $
for all $j \in \mathbb{N},\ k=1,...,2^j$ and all $\xi \in \varphi(\Gamma_{j,k}) ,$
we have that
\begin{align}\label{double integral: 4}
 \left( \int_{\mathbb{S}^1}   \mathcal{L}^1 (E_{2^{1-j}} (\xi)) ^{\frac{1}{p-1}} \, |d \xi| \right)^{p-1}
=&  \Big( \sum_{k=1}^{2^j} \int_{\varphi(\Gamma_{j,k})} \mathcal{L}^1 (E_{2^{1-j}} (\xi)) ^{\frac{1}{p-1}} \, |d \xi| \Big)^{p-1} \notag\\
\le &\bigg( \sum_{k=1}^{2^j} \ell(\varphi(\Gamma_{j,k})) \Big( \sum_{i= k-1} ^{k+1} \ell(\varphi(\Gamma_{j,i})) \Big)^{\frac{1}{p-1}} \bigg)^{p-1} \notag\\
\le & \sum_{k=1}^{2^j} \ell(\varphi(\Gamma_{j,k}))^{p-1} \sum_{i= k-1} ^{k+1} \ell(\varphi(\Gamma_{j,i})) .
\end{align}
Moreover by Young's inequality we have that
\begin{align}\label{double integral: 4-1}
\sum_{k=1}^{2^j} \ell(\varphi(\Gamma_{j,k}))^{p-1} \ell(\varphi(\Gamma_{j,k-1}))
\le & \sum_{k=1} ^{2^j} \frac{1}{p} \ell(\varphi(\Gamma_{j,k}))^{p} + \frac{p}{p-1} \ell(\varphi(\Gamma_{j,k-1}))^p \notag \\
\lesssim  & \sum_{k=1} ^{2^j} \ell(\varphi(\Gamma_{j,k}))^{p}
\end{align}
and 
\begin{align}\label{double integral: 4-2}
\sum_{k=1}^{2^j} \ell(\varphi(\Gamma_{j,k}))^{p-1} \ell(\varphi(\Gamma_{j,k+1}))
\le & \sum_{k=1} ^{2^j} \frac{1}{p} \ell(\varphi(\Gamma_{j,k}))^{p} + \frac{p}{p-1} \ell(\varphi(\Gamma_{j,k+1}))^p \notag \\
\lesssim  & \sum_{k=1} ^{2^j} \ell(\varphi(\Gamma_{j,k}))^{p}.
\end{align}
Combining \eqref{double integral: 4}, \eqref{double integral: 4-1} with \eqref{double integral: 4-2} implies that
\begin{equation}\label{double integral: 5}
\left( \int_{\mathbb{S}^1}   \mathcal{L}^1 (E_{2^{1-j}} (\xi)) ^{\frac{1}{p-1}} \, |d \xi| \right)^{p-1}
 \lesssim  \sum_{k=1} ^{2^j} \ell(\varphi(\Gamma_{j,k}))^{p} 
\end{equation}
for all $j \in \mathbb{N}.$
Let 
\begin{equation*}
\Lambda _{\lambda} (t) = 
\begin{cases}
t^{\lambda +1} & \lambda \neq -1 ,\\
\log t & \lambda =-1 .
\end{cases}
\end{equation*}
For any $j \in \mathbb{N},$
we have that 
\begin{equation}\label{double integral: 6-1}
\int_{2^{-j}} ^{2^{1-j}} t^{-1} \log^{\lambda} _{2} (t^{-1}) \ dt
\approx -\int_{2^{-j}} ^{2^{1-j}} \, d \Lambda_\lambda (\log _{2} (t^{-1}) )
= \Lambda_{\lambda}(j)-\Lambda_{\lambda}(j-1)
\approx j^{\lambda}.
\end{equation}
It follows \eqref{double integral: 6-1} and \eqref{dyadic deom 1-0} that
\begin{equation}\label{double integral: 6}
\int_{2^{-j}} ^{2^{1-j}} - \mathcal{A}'_{p,\alpha ,\lambda}(t) \, dt
\approx 2^{j(p-2-\alpha)} \int_{2^{-j}} ^{2^{1-j}} \frac{1}{t} \log^{\lambda} _{2} (t^{-1})\, dt
\approx \ell(\Gamma_{j,k})^{2+\alpha-p} j^{\lambda}.
\end{equation}
By combining \eqref{double integral: 3-1}, \eqref{double integral: 5} with \eqref{double integral: 6}, we conclude that
\begin{equation*}
\mathcal{V}^{\frac{1}{p-1}}  (p,\alpha,\lambda,\varphi) 
\lesssim \sum_{j=1} ^{+\infty} \sum_{k=1} ^{2^j} \ell(\varphi(\Gamma_{j,k}))^p  \ell(\Gamma_{j,k})^{2+\alpha-p} j^{\lambda} = \mathcal{E}_1 (p,\alpha,\lambda,\varphi).
\end{equation*}

We next consider the case $p \in [2,+\infty).$
By the analogous arguments as for \eqref{double integral: 3-1} we have that
\begin{equation}\label{double integral: 7}
\mathcal{V}^{\frac{1}{p-1}}  (p,\alpha,\lambda,\varphi) 
\gtrsim  \sum_{j=5} ^{+\infty} \int_{\pi 2^{1-j}} ^{\pi 2^{2-j}} -\mathcal{A}' _{p,\alpha ,\lambda}(t) \,  dt \left(\int_{\mathbb{S}^1}  \mathcal{L}^1 (E_{\pi 2^{1-j}} (\xi))^{\frac{1}{p-1}} \, |d \xi|   \right)^{ p-1 } .
\end{equation}
Since $\varphi(\Gamma_{j,k}) \subset E_{\pi 2^{1-j}} (\xi) 
$ for all $j \ge 5$, $ k \in \{1,..., 2^j\}$ and all $\xi \in \varphi(\Gamma_{j,k}) ,$ we have that 
\begin{align}\label{double integral: 8}
 \left(\int_{\mathbb{S}^1}   \mathcal{L}^1 (E_{\pi 2^{1-j}} (\xi))^{\frac{1}{p-1}} \, |d \xi|   \right)^{ p-1 } = 
 & \Big(\sum_{k=1} ^{2^j} \int_{\varphi(\Gamma_{j,k})}   \mathcal{L}^1 (E_{\pi 2^{1-j}} (\xi))^{\frac{1}{p-1}} \, |d \xi|  \Big)^{p-1} \notag\\
\ge &  \Big( \sum_{k=1} ^{2^j} \ell(\varphi(\Gamma_{j,k})) \ell(\varphi(\Gamma_{j,k}))^{\frac{1}{p-1}} \Big)^{p-1} \notag\\
\ge  & \sum_{k=1} ^{2^j} \ell(\varphi(\Gamma_{j,k}))^p.
\end{align}
By \eqref{double integral: 6}, \eqref{double integral: 7} and \eqref{double integral: 8}, there is a constant $C>0$ such that
\begin{equation*}
\mathcal{E}_1 (p,\alpha,\lambda,\varphi) =  \sum_{j=1} ^{4} \sum_{k=1} ^{2^j} + \sum_{j=5} ^{+\infty} \sum_{k=1} ^{2^j} 
\lesssim  C + \mathcal{V}^{\frac{1}{p-1}} (p,\alpha,\lambda,\varphi) . 
\end{equation*}
\end{proof}

We next prove Theorem \ref{main theorem} ($2$).

\begin{lemma}\label{lemma varphi chara}
Let $\varphi : \mathbb{S}^1 \rightarrow \mathbb{S}^1$ be a homeomorphism. For any $p \in (1,+\infty),\ \alpha \in (-1,p-1)$ and $\lambda \in \mathbb{R},$ we have that $\mathcal{U} (p,\alpha,\lambda,\varphi)$ and $\mathcal{E}_1 (p,\alpha,\lambda,\varphi)$ are comparable.
\end{lemma}

\begin{proof}
We first prove that $\mathcal{U} (p,\alpha,\lambda,\varphi)$ is controlled by $\mathcal{E}_1 (p,\alpha,\lambda,\varphi)$.
Given $j \ge 1 $ and $\xi \in \mathbb{S}^1 ,$ set
\begin{equation*}
A_j  = \{(\xi, \eta)\in \mathbb{S}^1 \times \mathbb{S}^1 : \pi 2^{-j} < \ell(\xi \eta) \le  \pi 2^{1-j} \}
\end{equation*}
and $A_j (\xi)= \{ \eta \in \mathbb{S}^1: (\xi ,\eta) \in A_j\} .$ Notice that $\lambda_{\mbd}$ is the Euclidean distance. 
We have that
\begin{align}\label{dbl itgl on previous: 0-1}
\mathcal{U} (p,\alpha,\lambda,\varphi)
=  & \sum_{j=1} ^{+\infty} \int_{A_j} \Phi_{p,\lambda} \Big(\frac{|\varphi(\xi) -\varphi(\eta)|}{|\xi -\eta|}  \Big) |\xi -\eta|^{\alpha} \, |d \eta| \, |d \xi| \notag \\
= & \sum_{j=1} ^{+\infty} \sum_{k=1} ^{2^j} \int_{\Gamma_{j,k}} \int_{A_j (\xi)}  \Phi_{p,\lambda} \Big(\frac{|\varphi(\xi) -\varphi(\eta)|}{|\xi -\eta|}  \Big) |\xi -\eta|^{\alpha} \, |d \eta| \, |d \xi| .
\end{align}
Notice that
\begin{equation}\label{dbl itgl on previous: 0}
|\xi -\eta| \approx \ell(\Gamma_{j,k}) \mbox{ and } |\varphi(\xi) -\varphi(\eta)| \le \sum_{i=k-1} ^{k+1} \ell(\varphi(\Gamma_{j,i})) \le 2\pi
\end{equation}
for all $ j \in \mathbb{N}$, $k \in \{1,..., 2^j\},$ $\xi \in \Gamma_{j,k}$ and $\eta \in A_j (\xi).$
It then follows that
\begin{align}\label{dbl itgl on previous: 1}
\Phi_{p,\lambda} \left(\frac{|\varphi(\xi) -\varphi(\eta)|}{|\xi -\eta|}  \right) |\xi -\eta|^{\alpha}
\lesssim  & \left(\sum_{i=k-1} ^{k+1} \ell(\varphi(\Gamma_{j,i})) \right)^p \ell(\Gamma_{j,k})^{\alpha-p} \log^{\lambda} (e+ 2 \pi \cdot 2^j) \notag \\
\lesssim &  \sum_{i=k-1} ^{k+1} \ell(\varphi(\Gamma_{j,i})) ^p \ell(\Gamma_{j,k})^{\alpha-p} j^{\lambda}
\end{align}
for all $\lambda \ge 0$, $\xi \in \Gamma_{j,k}$ and $\eta \in A_j (\xi).$
Since 
\begin{equation}\label{dbl itgl on previous: 1-1}
\mcl^1 (A_j (\xi)) \approx \ell(\Gamma_{j,k}) 
\end{equation}
for all $j \in \mathbb{N}$, $k=1,...,2^j$ and $\xi \in \Gamma_{j,k} ,$
we derive from \eqref{dbl itgl on previous: 0-1} and \eqref{dbl itgl on previous: 1} that
\begin{align*}
\mathcal{U} (p,\alpha,\lambda,\varphi)
\lesssim & \sum_{j=1} ^{+\infty} \sum_{k=1} ^{2^j} \sum_{i=k-1} ^{k+1} \ell(\varphi(\Gamma_{j,i})) ^p \ell(\Gamma_{j,k})^{\alpha-p} j^{\lambda}  \int_{\Gamma_{j,k}} \int_{A_j (\xi)} \, |d \eta|\, |d \xi| \notag\\
\lesssim & \sum_{j=1} ^{+\infty} \sum_{k=1} ^{2^j} \ell(\varphi(\Gamma_{j,k}))^p \ell(\Gamma_{j,k})^{\alpha-p} j^{\lambda} = \mathcal{E}_1 (p,\alpha,\lambda,\varphi)
\end{align*}
whenever $\lambda \ge 0.$

Since $\varphi$ is homeomorphic, for any $j \in \mathbb{N}$ and $k \in \{1,...,2^j\}$ there are $\xi'_{j,k} \in \Gamma_{j,k}$ and $\eta'_{j,k} \in A_j (\xi'_{j,k})$ such that
\begin{align}\label{dbl itgl on previous: 3-1}
& \Phi_{p,\lambda} \Big( \frac{|\varphi(\xi' _{j,k}) -\varphi(\eta' _{j,k})|}{|\xi' _{j,k} -\eta' _{j,k}|} \Big) |\xi' _{j,k} -\eta' _{j,k}|^{\alpha} \notag \\
= & \max \Big \{\Phi_{p,\lambda} \Big(\frac{|\varphi(\xi) -\varphi(\eta)|}{|\xi -\eta|}  \Big) |\xi -\eta|^{\alpha}: \xi \in \Gamma_{j,k} \mbox{ and } \eta \in A_j (\xi)\Big  \}.
\end{align}
Since $0<\alpha +1<p,$ there is $\beta \in (-1,0)$ such that $0<(1+\beta) p <\alpha+1.$
Define
\begin{equation*}\label{dbl itgl on previous: 3-2}
\chi(j,k)=
\begin{cases}
1 & \mbox{if } |\varphi(\xi'_{j,k}) -\varphi(\eta'_{j,k})| \le 2^{j \beta}, \\
0 & \mbox{otherwise}.
\end{cases}
\end{equation*}
From \eqref{dbl itgl on previous: 0-1}, \eqref{dbl itgl on previous: 3-1}, \eqref{dbl itgl on previous: 1-1}, \eqref{dyadic deom 1-0} and \eqref{dbl itgl on previous: 0}, we obtain that
\begin{align}\label{dbl itgl on previous: 3}
\mathcal{U} (p,\alpha,\lambda,\varphi)
\le & \sum_{j=1} ^{+\infty} \sum_{k=1} ^{2^j} \ell(\Gamma_{j,k})^{2+\alpha} \Phi_{p,\lambda} \Big(\frac{|\varphi(\xi' _{j,k}) -\varphi(\eta' _{j,k})|}{|\xi' _{j,k} -\eta' _{j,k}|} \Big)  \notag \\
=& \sum_{j=1} ^{+\infty} \sum_{k=1} ^{2^j} \ell(\Gamma_{j,k})^{2+\alpha} \Phi_{p,\lambda} \Big(\frac{|\varphi(\xi' _{j,k}) -\varphi(\eta' _{j,k})|}{|\xi' _{j,k} -\eta' _{j,k}|} \Big) \chi(j,k) \notag \\
&+ \sum_{j=1} ^{+\infty} \sum_{k=1} ^{2^j} \ell(\Gamma_{j,k})^{2+\alpha} \Phi_{p,\lambda} \Big(\frac{|\varphi(\xi' _{j,k}) -\varphi(\eta' _{j,k})|}{|\xi' _{j,k} -\eta' _{j,k}|} \Big) (1-\chi(j,k))
=: \sum\nolimits_1 +\sum\nolimits_2 .
\end{align}
Since $\log^{\lambda} (e+ |\varphi(\xi'_{j,k}) -\varphi(\eta'_{j,k})| |\xi'_{j,k} -\eta'_{j,k}|^{-1} ) \le 1 
$ for all $\lambda <0,$ $j \in \mathbb{N}$ and $1 \le k \le 2^j ,$
by \eqref{dbl itgl on previous: 0} and \eqref{dyadic deom 1-0} we have that
\begin{equation}\label{dbl itgl on previous: 4}
\sum\nolimits_1 \lesssim \sum_{j=1} ^{+\infty} 2^{-2j} \sum_{k=1} ^{2^j} 2^{j((1+\beta)p-\alpha)}
=  \sum_{j=1} ^{+\infty}2^{j((1+\beta)p-\alpha-1)} < +\infty.
\end{equation}
Moreover we derive from \eqref{dbl itgl on previous: 0} that
\begin{align}\label{dbl itgl on previous: 5}
\sum\nolimits_2 \lesssim & \sum_{j=1} ^{+\infty} \sum_{k=1} ^{2^j} \ell(\Gamma_{j,k})^{2+\alpha -p} \Big(\sum_{i=k-1} ^{k+1} \ell(\varphi(\Gamma_{j,k})) \Big)^p  \log^{\lambda} (2^{j(1+\beta)}) \notag \\
\lesssim & \sum_{j=1} ^{+\infty} \sum_{k=1} ^{2^j} \ell(\varphi(\Gamma_{j,k}))^p \ell(\Gamma_{j,k})^{2+\alpha -p}   j^{\lambda} = \mathcal{E}_1 (p,\alpha,\lambda,\varphi).
\end{align}
for all $\lambda < 0.$
Combining \eqref{dbl itgl on previous: 3}, \eqref{dbl itgl on previous: 4} with \eqref{dbl itgl on previous: 5} implies that 
there is a constant $C>0$ such that $\mathcal{U} (p,\alpha,\lambda,\varphi) \lesssim C+ \mathcal{E}_1 (p,\alpha,\lambda,\varphi)$
for all $\lambda <0.$

We next prove that $\mathcal{U} (p,\alpha,\lambda,\varphi)$ dominates $\mathcal{E}_1 (p,\alpha,\lambda,\varphi)$.
Given $\xi \in \mbs^1$ and $\eta \in \mbs^1 ,$ let $\ell (\xi \eta)$ be the arc length of the shorter arc in $\mbs^1$ connecting $\xi$ and $\eta .$ 
Given $j \ge 3$ and $\xi \in \mbs^1 ,$ set
\begin{equation*}
B_j  = \{(\xi, \eta)\in \mathbb{S}^1 \times \mathbb{S}^1 : \pi 2^{2-j} < \ell(\xi \eta) \le \pi 2^{3-j} \mbox{ with } \arg \eta > \arg \xi \}
\end{equation*}
and $B_j (\xi)= \{ \eta \in \mathbb{S}^1: (\xi ,\eta) \in B_j\}.$ We have that
\begin{equation}\label{dbl itgl on previous: 6}
\sum_{j=3} ^{+\infty} \sum_{k=1} ^{2^j} \int_{\Gamma_{j,k-1}} \int_{B_j (\xi)} \Phi_{p,\lambda} \Big( \frac{|\varphi(\xi) -\varphi(\eta)|}{|\xi -\eta|} \Big) |\xi -\eta|^{\alpha} \, |d \eta|\, |d \xi| = \mathcal{U} (p,\alpha,\lambda,\varphi).
\end{equation}
Since $\varphi$ is homeomorphic, for any $j \ge 3$ and $1 \le k \le 2^j$ there are $\xi'' _{j,k} \in \Gamma_{j,k-1}$ and $\eta'' _{j,k} \in B_{j} (\xi'' _{j,k})$ such that
\begin{align}\label{dbl itgl on previous: 6-1}
& \Phi_{p,\lambda} \Big( \frac{|\varphi(\xi'' _{j,k}) -\varphi(\eta'' _{j,k})|}{|\xi'' _{j,k} -\eta'' _{j,k}|}  \Big) |\xi'' _{j,k} -\eta'' _{j,k}|^{\alpha} \notag \\
= & \min \Big \{\Phi_{p,\lambda} \Big(\frac{|\varphi(\xi) -\varphi(\eta)|}{|\xi -\eta|} \Big) |\xi -\eta|^{\alpha}: \xi \in \Gamma_{j,k-1} \mbox{ and } \eta \in B_j (\xi)\Big \}.
\end{align}
Notice that 
\begin{equation}\label{dbl itgl on previous: 7}
|\xi'' _{j,k} -\eta'' _{j,k}| \approx \ell(\Gamma_{j,k}) \mbox{ and } 2\ge |\varphi(\xi'' _{j,k}) -\varphi(\eta'' _{j,k})|  \gtrsim  \ell(\varphi(\Gamma_{j,k}))
\end{equation}
whenever $j \ge 3$ and $k \in \{1,...,2^j\}.$
Since $\mcl^1 (B_j (\xi)) \approx \ell (\Gamma_{j,k})$ for all $j \ge 3,\ k=1,...,2^j$ and $\xi \in \mathbb{S}^1,$
it follows from \eqref{dbl itgl on previous: 6}, \eqref{dbl itgl on previous: 6-1} and \eqref{dbl itgl on previous: 7} that
\begin{equation}\label{dbl itgl on previous: 8}
\sum_{j=3} ^{+\infty}  \sum_{k=1} ^{2^j} \ell(\Gamma_{j,k})^{2+\alpha} \Phi_{p,\lambda} \Big(\frac{|\varphi(\xi'' _{j,k}) -\varphi(\eta'' _{j,k})|}{|\xi'' _{j,k} -\eta'' _{j,k}|}  \Big) |
\lesssim \mathcal{U} (p,\alpha,\lambda,\varphi) .
\end{equation}
Moreover, for any $\lambda \le 0 $ we obtain from \eqref{dbl itgl on previous: 7} that
\begin{equation}\label{dbl itgl on previous: 8-2}
j^{\lambda}
\lesssim \log^{\lambda} ( e+ 2^{1+j})
\lesssim \log^{\lambda} \Big( e+ \frac{|\varphi(\xi'' _{j,k}) -\varphi(\eta'' _{j,k})|}{|\xi'' _{j,k} -\eta'' _{j,k}|} \Big)
\end{equation}
for all $j \in \mathbb{N}$ and all $k =1,...,2^j .$
From \eqref{dbl itgl on previous: 7}, \eqref{dbl itgl on previous: 8} and \eqref{dbl itgl on previous: 8-2}, 
there is a constant $C>0$ such that
\begin{equation*}
\mathcal{E}_1 (p,\alpha,\lambda,\varphi) =C+ \sum_{j=3} ^{+\infty} \sum_{k=1} ^{2^j} \ell(\varphi(\Gamma_{j,k}))^p \ell(\Gamma_{j,k})^{2+\alpha -p} j^{\lambda}
\lesssim C+
\mathcal{U} (p,\alpha,\lambda,\varphi)
\end{equation*}
for all $\lambda \le 0.$
For any $\lambda>0,$ by \eqref{dbl itgl on previous: 7} and \eqref{dbl itgl on previous: 8} there is a constant $C>0$ such that
\begin{equation}\label{dbl itgl on previous: 9}
 \sum_{j=1} ^{+\infty}\sum_{k=1} ^{2^j} \ell(\varphi(\Gamma_{j,k}))^p  \ell(\Gamma_{j,k})^{2+\alpha-p} \log^{\lambda} \big( 2^j\ell(\varphi(\Gamma_{j,k}))\big) 
 \lesssim C+  \mathcal{U} (p,\alpha,\lambda,\varphi) .
\end{equation}
Let $\beta$ be same as in \eqref{dbl itgl on previous: 3-2}. Set
\begin{equation*}
\chi_{j,k} = 
\begin{cases}
1 & \mbox{if } \ell(\varphi(\Gamma_{j,k})) \le 2^{j \beta}, \\
0 & \mbox{otherwise}.
\end{cases}
\end{equation*}
We have that
\begin{align}\label{dbl itgl on previous: 10}
\mathcal{E}_1 (p,\alpha,\lambda,\varphi) =& \sum_{j=1} ^{+\infty} \sum_{k=1} ^{2^j} \ell(\varphi(\Gamma_{j,k}))^p  \ell(\Gamma_{j,k})^{2+\alpha-p} j^{\lambda} \chi_{j,k} \notag \\
&+ \sum_{j=1} ^{+\infty} \sum_{k=1} ^{2^j} \ell(\varphi(\Gamma_{j,k}))^p  \ell(\Gamma_{j,k})^{2+\alpha-p} j^{\lambda} (1-\chi_{j,k})   =: \sum \nolimits^1 +\sum\nolimits^2 .
\end{align}
Moreover
\begin{equation}\label{dbl itgl on previous: 11}
\sum \nolimits^1 \le \sum_{j=1} ^{+\infty} 2^{j((1+\beta)p-\alpha-1)} j ^{\lambda} <\infty
\end{equation}
and 
\begin{equation}\label{dbl itgl on previous: 12}
\sum\nolimits^2
\lesssim  \sum_{j=1} ^{+\infty}\sum_{k=1} ^{2^j} \ell(\varphi(\Gamma_{j,k}))^p  \ell(\Gamma_{j,k})^{2+\alpha-p} \log^{\lambda} \big( 2^j\ell(\varphi(\Gamma_{j,k}))\big) .
\end{equation}
From \eqref{dbl itgl on previous: 10}, \eqref{dbl itgl on previous: 11}, \eqref{dbl itgl on previous: 12} and \eqref{dbl itgl on previous: 9},
we have that $\mathcal{U} (p,\alpha,\lambda,\varphi)$ controls $\mathcal{E}_1 (p,\alpha,\lambda,\varphi)$ whenever $\lambda >0.$
\end{proof}

\begin{proof}[\bf{Proof of Theorem \ref{main theorem} ($2$)}]
By Lemma \ref{connect mathcal E_1 and mathcal E_2}, Lemma \ref{I_1 discrete} and Lemma \ref{desvribe I_2}, for any $p>1$ we have that both $I_1 (p,\alpha ,\lambda ,h)$ and $I_2 (p,\alpha ,\lambda ,h)$ are comparable to $\mathcal{E}_1 (p,\alpha,\lambda,\varphi)$ whenever $\alpha \in (-1,p-1)$ and $\lambda \in \mathbb{R}.$
By Lemma \ref{lemma varphi chara}, we hence conclude comparability of both $I_1 (p,\alpha ,\lambda ,h)$ and $I_2 (p,\alpha ,\lambda ,h)$ with $\mathcal{U} (p,\alpha,\lambda,\varphi)$ for all $p >1, \alpha \in (-1,p-1)$ and every $\lambda \in \mathbb{R}.$
By Lemma \ref{lemma: partial characterization}, we can dominate $\mathcal{V} (p,\alpha,\lambda,\varphi)$ by either $I_1 (p,\alpha ,\lambda ,h)$ or $I_2 (p,\alpha ,\lambda ,h)$ whenever $p \in (1,2],$ while both $I_1 (p,\alpha ,\lambda ,h)$ and $I_2 (p,\alpha ,\lambda ,h)$ are controlled by $\mathcal{V} (p,\alpha,\lambda,\varphi)$ for all $p \in [2,+\infty).$ Moreover from Example \ref{example 1} and Example \ref{example 2}, we have that $\mathcal{V} (p,\alpha,\lambda,\varphi)$ is comparable to either $I_1 (p,\alpha ,\lambda ,h)$ or $I_2 (p,\alpha ,\lambda ,h)$ only when $p=2.$ 
\end{proof}

On the proof of Theorem \ref{main theorem} ($3$), we have the following general result.

\begin{lemma}\label{lem3.7}
Let $\Omega \subset \mbr^2$ be a Jordan domain and $\varphi : \mbs^1 \rightarrow \partial \Omega$ be a homeomorphism. 
For any $p>1,$ there is no a homeomorphic extension $h :\mbd \rightarrow \Omega$ of $\varphi$ for which $I_1 (p,\alpha ,\lambda ,h)<+\infty$ for either $\alpha \in (-\infty, -1)$ and $\lambda \in \mathbb{R}$ or $\alpha=-1$ and $\lambda \in [-1,+\infty);$ and for which $I_2 (p,\alpha ,\lambda ,h)<+\infty$ for all $\alpha \in (-\infty, -1]$ and $\lambda \in \mathbb{R}.$
\end{lemma}

\begin{proof}
Assume that there is a homeomorphic extension $h :\mbd \rightarrow \Omega$ of $\varphi$ for which $I_1 (p,\alpha ,\lambda ,h)<+\infty$ for either $\alpha \in (-\infty, -1)$ and $\lambda \in \mathbb{R}$ or $\alpha=-1$ and $\lambda \in [-1,+\infty).$ Then $h \in W^{1,p}(\mbd , \Omega) .$ Let 
\begin{equation*}
\mathbb{S}_r = \{\xi \in \mathbb{R}^2 : |\xi| =r \} \mbox{ and }
\mbox{osc}_{\mathbb{S}_r}  h = \sup\{ |h (\xi_1) -  h(\xi_2)|: \xi_1 , \xi_2 \in \mathbb{S}_r\} .
\end{equation*}
By the ACL-property of Sobolev mappings, we have that
\begin{equation}\label{the last case: 1}
\mbox{osc}_{\mathbb{S}_r} h \le \int_{\mbs_r}   |D h(\xi) |  \, |d \xi| 
\end{equation}
for $\mcl^1$-a.e. $r \in [0,1).$  By Jensen's inequality we derive from \eqref{the last case: 1} that
\begin{align}\label{the last case: 2}
(\mbox{osc}_{\mathbb{S}_r}  h )^p 
\le & (\mbox{osc}_{\mathbb{S}_r} h )^p r^{1-p}
\lesssim  \int_{\mbs_r}     |D h(\xi) |^p  \, |d \xi| \notag \\
= & w^{-1} _{\alpha ,\lambda} (1-r) \int_{\mbs_r}   |D h(\xi) |^p  w_{\alpha ,\lambda} (1-r)\, |d \xi| .
\end{align}
Let $ \mbd_r = \{z \in \mbr^2 : |z|<r\}.$
Since $h$ is a homeomorphism, we have $\mbox{osc}_{\mbd_r} h  = \mbox{osc}_{\mbs_r} h .$ Hence
\begin{equation}\label{the last case: 4-0}
\mbox{osc}_{\mathbb{S}_r} h \mbox{ is increasing with respect to } r \in [0,1) .
\end{equation}
Moreover $w_{\alpha,\lambda} (1-r) \approx 2^{-\alpha j} j^{\lambda}$ for all $j \ge 0 \mbox{ and } r \in (1-2^{-j}, 1-2^{-j-1}].$ 
By \eqref{the last case: 2}, \eqref{the last case: 4-0} and Fubini's theorem, we obtain that
\begin{align}\label{the last case: 4}
\sum_{j=1} ^{+\infty}  (\mbox{osc}_{\mathbb{S}_{1-2^{-j}}} h )^p  2^{-(\alpha+1) j} j^{\lambda}
\le & \sum_{j=1} ^{+\infty} \int_{1-2^{-j}} ^{1-2^{-j-1}} (\mbox{osc}_{\mathbb{S}_r} h)^p w_{\alpha ,\lambda} (1-r) \, dr  \notag \\
\lesssim & \sum_{j=1} ^{+\infty} \int_{1-2^{-j}} ^{1-2^{-j-1}} \int_{\mbs_r}    |D h(\xi) |^p   w_{\alpha ,\lambda} (1-r)  \, |d \xi | \, d r \notag\\
=  &  I_1 (p,\alpha ,\lambda ,h).
\end{align}
By the assumption at the beginning, we derive from \eqref{the last case: 4} that
\begin{equation}\label{the last case: 5}
\sum_{j=1} ^{+\infty}  (\mbox{osc}_{\mathbb{S}_{1-2^{-j}}}  h )^p  2^{-(\alpha+1) j} j^{\lambda} <+\infty
\end{equation}
for either $\alpha<-1$ and $\lambda \in \mathbb{R}$ or $\alpha=-1$ and $\lambda \ge -1.$
Hence by \eqref{the last case: 4-0} we have that $\mbox{osc}_{\mathbb{S}_{1-2^{-j}}} h  =0$ for all $j \ge 1.$ 
Therefore there is a constant $C$ such that $h(z)=C$ for all $z \in \mathbb{D}.$ This contradicts the homeomorphicity of $h$. We conclude that the assumption at the beginning cannot hold. 

We next assume that there is a homeomorphic extension $h :\mbd \rightarrow \Omega$ of $\varphi$ for which $I_2 (p,\alpha ,\lambda ,h)<+\infty$ for all $\alpha \in (-\infty, -1]$ and $\lambda \in \mathbb{R}.$ It is not difficult to see that $h \in W^{1,1} (\mbd ,\Omega).$ We first let $\lambda \ge 0 .$ The property (P-1) of Proposition \ref{proposition of Phi: lamdba >0} shows that $\Phi_{p,\lambda}$ is convex. Analogously to \eqref{the last case: 4}, we have
\begin{equation}\label{the last case: 6}
\sum_{j=1} ^{+\infty}  \Phi_{p,\lambda}\left(\frac{\mbox{osc}_{\mathbb{S}_{1-2^{-j}}} \mbox{Re} h }{2 \pi}  \right)  2^{-(\alpha+1) j} \lesssim I_2 (p,\alpha ,\lambda ,h).
\end{equation}
Analogous arguments as below \eqref{the last case: 5} imply that there is a contradiction under the above assumption.
We next let $\lambda <0.$ Proposition \ref{Psi_p,lambda} shows that $\Psi_{p, \lambda}$ is convex.
Analogously to \eqref{the last case: 6}, we obtain from \eqref{remark 1-1} that 
\begin{align*}\label{the last case: 7}
\sum_{j=1} ^{+\infty}  \Psi_{p,\lambda} \left(\frac{\mbox{osc}_{\mathbb{S}_{1-2^{-j}}} \mbox{Re} h  }{2 \pi} \right)  2^{-(\alpha+1) j}
\lesssim & \int_{\mathbb{D}} \Psi_{p,\lambda}(|Dh(z)|)w_{\alpha ,0} (z)\, dz \notag \\
\approx & I_2 (p,\alpha ,\lambda ,h).
\end{align*}
Analogous arguments as below \eqref{the last case: 5} imply that there is a contradiction under the above assumption.
\end{proof}

\begin{proof}[\bf{Proof of Theorem \ref{main thm}}]
Let $\lambda_{\Omega}$ be the internal distance and $|\cdot|$ be the Euclidean distance.
As the proof of \cite[Theorem 1]{our paper} shows that there exist a bi-Lipschitz mapping $g: (\mbs^1, |\cdot|) \rightarrow (\partial \Omega ,\lambda_{\Omega})$ and a diffeomorphic bi-Lipschitz extension $\tilde{g} : (\mbd, |\cdot|) \rightarrow (\Omega ,\lambda_{\Omega})$ of $g.$
Let $h = \tilde{g} \circ P[g^{-1} \circ \varphi].$ Then $h : \mbd \rightarrow \Omega$ is a diffeomorphic extension of $\varphi .$ Moreover
\begin{equation*}
I_1 (p,\alpha, \lambda, h) \approx I_1 (p,\alpha, \lambda, P[g^{-1} \circ \varphi])
,\ 
I_2 (p,\alpha, \lambda, h) \approx I_2 (p,\alpha, \lambda, P[g^{-1} \circ \varphi]) ,
\end{equation*}
\begin{equation*}
\mathcal{U}(p,\alpha ,\lambda , \varphi) \approx \mathcal{U} (p,\alpha ,\lambda , g^{-1} \circ \varphi)
,\ 
\mathcal{V} (p,\alpha ,\lambda , \varphi) \approx \mathcal{V} (p,\alpha ,\lambda , g^{-1} \circ \varphi)  .
\end{equation*}
Hence Theorem \ref{main thm} ($1$) and ($2$) follow from Theorem \ref{main theorem}. By Lemma \ref{lem3.7}, we complete the proof of Theorem \ref{main thm}.
\end{proof}

\section{Examples}
In this section, we give examples related to Theorem \ref{main theorem} ($2$).
We first decompose $[0,1] .$ 
For a given $s \in (0, +\infty),$ let 
\begin{equation}\label{j^s _n}
j^s _n = [2^{\frac{n}{s}}]
\end{equation}
be the largest integer less than $2^{n/s}.$
There is $n^s _0   \ge 1$ such that
\begin{equation}\label{n_0}
2^{-2-j^s _n} \ge 2^{-j^s _{n+1}} \mbox{ and } 2^{-j^s _n} \le 4^{-n} \qquad \forall n \ge n^s _0 -1.
\end{equation}

$\mathit{Step\ 1}$.
Let
\begin{equation*}
I_1 = I_{1,1} = (a_{1,1},a_{1,2}) \qquad \mbox{where }a_{1,1}=4^{-1} \mbox{ and }a_{1,2}=1- 4^{-1}.
\end{equation*}
Renumber the elements in $T_{1} = \{ 0,1\} \cup \partial I_1$ as $\{b_{1,i_{1}}: i_{1}=1,..., 4\}$ such that $b_{1,i'_{1}} < b_{1,i''_{1}}$ if $i'_{1} < i''_{1}.$

$\mathit{Step\ 2}$.
Let 
\begin{equation*}
I_{2,1} = (b_{1,1}+ 4^{-2}, b_{1,2}-4^{-2}) \mbox{ and } I_{2,2} = (b_{1,3}+ 4^{-2}, b_{1,4}-4^{-2}) .
\end{equation*}
Set $I_2 = \cup_{i=1} ^{2} I_{2,i}$, and renumber the elements in $T_{2} =  T_{1} \cup \partial I_2$ as $\{b_{2,i_{2}}: i_{2} =1,...,8\}$ such that $b_{2,i'_{2}} < b_{2,i''_{2}}$ if $i'_{2} < i''_{2}.$

After Step (n-$1$), we have $\{I_{n-1,k_{n-1}}: k_{n-1} =1,...,2^{n-2}\},$ $I_{n-1} = \cup_{k_{n-1} = 1 } ^{2^{n-2}} I_{n-1,k_{n-1}}$ and 
$T_{n-1} :=  T_{n-2} \cup \partial I_{n-1}=\{b_{n-1,i_{n-1}}:i_{n-1} = 1, ..., 2^n\}$
where $b_{n-1,i'_{n-1}} < b_{n-1,i''_{n-1}}$ if $i'_{n-1} < i''_{n-1} .$
In the following Step n, set
\begin{equation}\label{definition of I_n,k_n}
I_{n,k_n}:= (b_{n-1,2k_{n}-1 } + 4^{-n}, b_{n-1,2k_{n}} - 4^{-n}) \qquad \mbox{for } k_n = 1,..., 2^{n-1} .
\end{equation}
and $I_n = \cup_{k_{n} = 1 } ^{2^{n-1}} I_{n,k_{n}} .$
After renumbering the elements in $T_{n} = T_{n-1} \cup \partial I_n$ as above, we can proceed to Step (n+$1$).
Moreover we must replace $I_{n,k_n}$ in \eqref{definition of I_n,k_n} by
\begin{equation}\label{replacement of I_n,k_n}
I_{n,k_n}= (b_{n-1,2k_{n}-1 } + 2^{-j^s _n}, b_{n-1,2k_{n}} - 2^{-j^s _n})
\end{equation}
whenever $n \ge n^s _0 .$ 
Let $I = \cup_{n=1} ^{\infty} I_n$ and $R = [0,1] \setminus I .$
Then $R \neq \emptyset .$
We finally decompose $[0,1]$ as
\begin{equation}\label{decomposition of [0,1]}
R \cup I .
\end{equation}

We next give an estimate on the length of $I_{n,k_n}.$ 
Since 
$\mcl^1(I_{n,k_n}) = 2^{-j_{n-1}}  - 2^{1-j_{n}}$
for all $n \ge n_0  +1$ and $k_n \in \{1,...,2^{n-1}\} ,$
by the first inequality in \eqref{n_0} we have that
\begin{equation}\label{leng 1}
\mcl^1 (I_{n,k_n}) \ge 2^{-1-j_{n-1}}.
\end{equation}
for all $n \ge n_0  +1$ and $k_n \in \{1,...,2^{n-1}\}.$
When $n =n_0 ,$ from \eqref{replacement of I_n,k_n} and the second estimate in \eqref{n_0} we have that
\begin{equation}\label{leng 2}
\mcl^1(I_{n,k_n}) = 4^{1- n_0 } - 2^{1-j_{n_0}} \ge 4^{-n_0  +1/2} >4^{-n_0}
\end{equation}
for all $k_n = 1,...,2^{n-1} .$
Whenever $1\le n \le n_0 -1$ and $k_n \in \{1,...,2^{n-1}\},$ we have $\mcl^1 (I_{n,k_n}) = 4^{- n}.$
Let $C_1 (s)= \min \{2^{j_{n-1} -2n} : 1 \le n \le n_0\}.$ 
Then
\begin{equation}\label{leng 3}
\mcl^1 (I_{n,k_n}) \ge C_1 (s) 2^{-j_{n-1}}
\end{equation}
for all $1 \le n \le n_0$ and $k_n \in \{1,...,2^{n-1}\}.$
By \eqref{leng 1}, \eqref{leng 2} and \eqref{leng 3}, we obtain that there is a constant $C(s) >0$ such that
\begin{equation}\label{estimate of |I_n,k_n|}
\mcl^1 (I_{n,k_n}) \ge C(s) 2^{-j_{n-1}}
\end{equation}
for all $n \in \mathbb{N}$ and $k_n \in \{1,...,2^{n-1}\}.$

Define
\begin{equation}\label{definition of f^1}
f_{n, s} ^{1} (x) = \sum_{k_{n} = 1} ^{2^{n-1}} \frac{2k_{n}-1}{2^n} \chi_{\overline{I_{n,k_n}}} (x) \mbox{ and } f^{1} _s (x)= \sum_{n=1} ^{+\infty} f_{n, s} ^{1} (x).
\end{equation}
For any $x \in R$ and any $n \ge n^s _0 ,$ there is $b_n \in \partial I_n$ such that $|b_n -x| = \inf\limits_{b \in \partial I_n} |b-x|.$ 
By \eqref{replacement of I_n,k_n} and \eqref{definition of f^1}, we have that
\begin{equation*}\label{definition of f-1}
|b_n -x | \le 2^{-j_n} \mbox{ and } |f^{1} _s (b_{n+1})-f^{1} _s (b_n)| < 2^{-n-1}.
\end{equation*}
It follows that $\lim_{n \rightarrow +\infty} b_n =x$ and $\{f^{1} (b_n) \}$ is a Cauchy sequence.
Therefore 
\begin{equation}\label{definition of f}
f_s(x)=
\begin{cases}
f^{1} _s (x) &\mbox{if } x \in I ,\\
\lim\limits_{n \rightarrow +\infty} f^{1} _s(b_n) &\mbox{if } x \in R.
\end{cases}
\end{equation}
is a well-defined function on $[0,1] .$

\begin{proposition}\label{property of f}
Let $f_s$ be as in \eqref{definition of f} with $s \in (0,+\infty).$ Then $f_s(0)=0$, $f_s(1)=1$ and $f_s$ is increasing on $[0,1] .$ Moreover there is a constant $C(s)>0$ such that
\begin{equation}\label{modulus of continuity: f}
|f(x) -f(y)| \log^s (|x-y|^{-1}) \le C(s)
\end{equation}
for all $x, y \in [0,1]$ with $x \neq y.$
\end{proposition}

\begin{proof}
By \eqref{definition of f}, we have that $f_s (0) = \lim_{n \rightarrow \infty} f^1 _s  (2^{-j_n}) = \lim_{n \rightarrow \infty} 2^{-n} =0.$
Analogously $f_s (1) =1 .$

We next prove the monotonicity of $f_s$. Let $x_1 \in [0,1],\ x_2 \in [0,1]$ with $x_1 \le x_2.$
If $x_1 \in I_{n,k'_n}$ and $ x_2 \in I_{n,k'' _n}$ with $k'_n \le k'' _n,$ from \eqref{definition of f} we have that
\begin{equation}\label{case 1}
f_s (x_1) \le f_s (x_2). 
\end{equation}
Assume $x_1 \in I_{n_1,k_{n_1}}$ and $ x_2 \in I_{n_2,k_{n_2}}$ with $n_1 \neq n_2.$
Let $q=|n_2 - n_1|.$
If $n_1 < n_2 ,$ from the construction of $\{I_{n,k_n}\}$ we have that $k_{n_{2}} \ge 2^q (k_{n_{1}} -1) +2^{q-1} +1.$
It then follows from \eqref{definition of f^1} that
\begin{equation}\label{1case 2-1}
f_s (x_2) \ge \frac{2 \big(2^q (k_{n_{1}} -1) +2^{q-1} +1 \big) -1}{2^{n_1} 2^q} > f_s (x_1).
\end{equation}
If $n_2 < n_1,$ from the construction of $\{I_{n,k_n}\}$ we have that
\begin{equation*}\label{monot: 1-0}
k_{n_{2}} \ge
\begin{cases}
\left[ \frac{k_{n_1} }{ 2^q }\right] +1 & \mbox{if }  0 \le \frac{k_{n_1} }{ 2^q }  - \left[ \frac{k_{n_1} }{ 2^q } \right] \le 1/2 ,\\
\left[ \frac{k_{n_1} }{ 2^q } \right] +2 & \mbox{if } 1/2 < \frac{k_{n_1} }{ 2^q }  - \left[\frac{k_{n_1} }{ 2^q }  \right] <1.
\end{cases}
\end{equation*}
It follows that
\begin{equation}\label{monot: 1}
2k_{n_2} -1 \ge 2( \frac{k_{n_1} }{ 2^q } + 1/2) -1= 2 \frac{k_{n_1} }{ 2^q } \qquad \mbox{if } 0 \le \frac{k_{n_1} }{ 2^q }  - \left[ \frac{k_{n_1} }{ 2^q }  \right] \le 1/2
\end{equation}
and
\begin{equation}\label{monot: 2}
2k_{n_2} -1 \ge 2(\frac{k_{n_1} }{ 2^q } +1) -1 = 2 \frac{k_{n_1} }{ 2^q } +1 \qquad \mbox{if } 1/2 < \frac{k_{n_1} }{ 2^q }  - \left[ \frac{k_{n_1} }{ 2^q }  \right] <1.
\end{equation}
By combining \eqref{monot: 1} with \eqref{monot: 2}, we deduce from \eqref{definition of f^1} that 
\begin{equation}\label{1case 2-2}
f_s (x_2) > f_s (x_1).
\end{equation}
Assume $x_1 \in R$ and $x_2 \in I.$ By \eqref{definition of f}, there is $\{ b_n \} \subset \partial I$ such that $\lim\limits_{n \rightarrow \infty} b_n =x_1.$
Together with $x_1 < x_2,$ it follows that $b_n < x_2$ whenever $n \gg 1.$ Via the arguments for \eqref{case 1}, \eqref{1case 2-1} and \eqref{1case 2-2}, we have that
\begin{equation}\label{monot: 2-1}
f^{1} _s (b_n) \le f_s(x_2) \qquad \forall n \gg 1.
\end{equation}
By taking limit for \eqref{monot: 2-1}, we have that
\begin{equation}\label{case 3}
f_s(x_1) \le f_s (x_2).
\end{equation}
Assume either $x_1 \in I$ and $x_2 \in R,$ or $x_1 \in R$ and $x_2 \in R.$
Via analogous arguments as for \eqref{case 3}, we can also prove $f_s(x_1) \le f_s(x_2)$ at these two cases.
By preceding arguments, we conclude that $f_s$ is increasing on $[0,1]$.

We next prove \eqref{modulus of continuity: f}.
Let $T_n = \{b_{n, i_n} : i_n =1,...,2^{n+1}\}$ with $n \in \mathbb{N}$ and $f^1 _{i,s}$ be as in \eqref{definition of f^1}.
For a given $n \in \mathbb{N},$
define 
\begin{equation*}
\  f_{n , s} ^{2} (x) = \sum_{i=1} ^{2^n} \left( \frac{2^{j_n}}{2^n} (x-b_{n,2i-1})+ \frac{i-1}{2^n}\right) \chi_{[b_{n,2i-1},b_{n,2i}]} (x),
\end{equation*}
\begin{equation}\label{definition of f_n}
f_{n ,s}(x) = f_{n ,s} ^{2} (x) + \sum_{i=1} ^{n} f_{i ,s} ^{1} (x)  .
\end{equation}
Then $f_{n,s}$ is piecewise affine, increasing and continuous on $[0,1] .$
Furthermore we claim:
\begin{enumerate}
\item[(i)] $\lim_{n \rightarrow \infty} f_{n,s } (x_0) =f_s(x_0)$ for all $x_0 \in [0,1],$
\item[(ii)] there are constant $C(s)>0$ and $N(s) >0$ such that
\begin{equation*}
\sup \big\{ |f_{n,s} (x) - f_{n,s} (y)| \log ^{s} (|x-y|^{-1} ):  x,y \in [0,1] \mbox{ and } x \neq y\big\} \le C(s)
\end{equation*}
for all $n \ge N(s).$
\end{enumerate}
If both (i) and (ii) hold, we can prove \eqref{modulus of continuity: f}.

We first prove (i). Let $x_0 \in [0,1].$
If $x_0 \in I, $ without loss of generality we assume that $x_0 \in I_{n_0,k_{n_0}}.$ From \eqref{definition of f} and \eqref{definition of f_n}, we have that $f_n (x_0) = f(x_0)$ for all $n \ge n_0.$
Therefore (i) holds.
If $x_0 \in R$, from \eqref{definition of f} there is $\{ b_n\} \subset \partial I$ such that $\lim_{n \rightarrow \infty} b_n =x_0$ and $\lim_{n \rightarrow \infty} f^1 _s (b_n) =f_s (x_0).$
Moreover by \eqref{definition of f_n}, we have that
\begin{equation*}
|f_{n,s} (x_0) -f^1 _s (b_n)| = |f_{n,s} (x_0)-f_{n,s} (b_n)  | \le 2^{-n}.
\end{equation*}
Together with $|f_{n,s} (x_0) - f_s (x_0)| \le |f_{n,s} (x_0) - f^1 _s (b_n)| + |f^1 _s(b_n) - f_s (x_0)|,$ we have that (i) also holds at this case.

We next prove (ii). Given $n \ge 1,\ x \in [0,1]$ and $y \in [0,1] $ with $x <y ,$
set
\begin{equation*}
k_n (x,y) = \# \{I_{m,k_m}:\ I_{m,k_m} \subset [x,y]\mbox{ for }m =1,..., n \mbox{ and } k_m  = 1,..., 2^{m-1}\}.
\end{equation*}
Then $ 0 \le k_n (x,y) \le 2^n -1.$

Assume $k_n (x,y)=0$. 
If $x \in \cup_{m=1}^{n} I_m,$
there are $m \in \{1,..., n\}$ and $k_m  \in \{1,..., 2^{m-1}\}$ such that $x \in I_{m ,k_m} .$
For the location of $y,$ possibly we have that
\begin{equation}
y \in I_{m,m_k},\ y \in I_{m,m_k +1},\mbox{ or } y \in [0,1] \setminus (\cup_{m=1}^{n} I_m).
\end{equation}
If $y \in I_{m,m_k}$, by \eqref{definition of f_n} we have that 
\begin{equation*}\label{estimate on 1/2^ j_n -3}
f_{n,s} (x) =f_{n,s}(y) \qquad \forall n \ge m .
\end{equation*} 
If $y \in I_{m,m_k +1} ,$ then $|x-y| \ge 2^{-j_n}.$ It follows from \eqref{definition of f_n} that
\begin{equation}\label{estimate on 1/2^ j_n -2}
|f_{n,s} (x) -f_{n,s} (y)| \log^s (|x-y|^{-1}) \le 2^{-n} \log^s (2^{j_n}) < 1.
\end{equation}
If $y \in [0,1] \setminus (\cup_{m=1}^{n} I_m)$, there is $x_0 \in [x,y) \cap T_n$ such that
\begin{equation}\label{estimate on 1/2^ j_n -1}
 0< y-x_0 < 2^{-j_n} \mbox{ and }f_{n,s}(x)=f_{n,s} (x_0).
\end{equation}
Since there is $n^s _{1}  >0$ such that $\log(2^{j^s _{n^s _1}}) -s >0 ,$  
we have that
\begin{equation}\label{estimate on 1/2^ j_n }
t \log^s (t^{-1}) \le 2^{-j^s _n} \log^s (2^{j^s _n}) < 2^{n-j_n}
\end{equation}
for all $n \ge n^s _1 $ and every $t \in (0,2^{-j^s _n}].$
By \eqref{definition of f_n}, \eqref{estimate on 1/2^ j_n -1} and \eqref{estimate on 1/2^ j_n }, we then have that
\begin{align}\label{case 1-1-3}
|f_{n,s} (x) -f_{n,s} (y)| \log^s (|x-y|^{-1})
\le   &\frac{|f_{n,s} (x_0) -f_{n,s} (y)|}{|x_0 - y|} |x_0 - y| \log^s (|x_0 -y|^{-1})  \notag \\
< & \frac{2^{j_n}}{2^n} \frac{2^n}{2^{j_n}} =1
\end{align}
whenever $n \ge N(s):=\max\{n^s _0  , n^s _1  \}.$
If $x \in [0,1] \setminus (\cup_{m=1}^{n} I_m),$ for the location of $y$ we possibly have that 
\begin{equation*}
y \in [0,1] \setminus (\cup_{m=1}^{n} I_m) ,\ y \in \cup_{m=1}^{n} I_m .
\end{equation*}
If $y \in [0,1] \setminus (\cup_{m=1}^{n} I_m),$
then $0< y-x < 2^{-j_n}.$ By \eqref{definition of f_n} and \eqref{estimate on 1/2^ j_n } we have that 
\begin{equation}\label{case 1-2-1}
|f_{n,s} (x) -f_{n,s} (y)| \log^s (|x-y|^{-1} ) = \frac{2^{j_n}}{2^n} |x-y| \log^s (|x-y|^{-1}) < 1
\end{equation}
for all $n \ge N(s).$
If $y \in \cup_{m=1}^{n} I_m ,$
by analogous arguments as for \eqref{case 1-1-3}
we have that
\begin{equation}\label{case 1-2-2}
|f_{n,s} (y) - f_{n,s} (x)| \log^s (|x-y|^{-1}) <1 
\end{equation}
for all $n \ge N(s).$
By \eqref{estimate on 1/2^ j_n -3}, \eqref{estimate on 1/2^ j_n -2}, \eqref{case 1-1-3}, \eqref{case 1-2-1} and \eqref{case 1-2-2}, 
we conclude that
\begin{equation}\label{case 11}
|f_{n,s} (y) - f_{n,s} (x)| \log^s (|x-y|^{-1}) <1 
\end{equation}
for all $n \ge N(s) \mbox{ and } k_n (x,y)=0.$

Assume $k_n (x,y) \in \{1,...,  2^n -1\}.$
Define
\begin{equation*}
x' = \inf \{e \in I_{m,k_m}:  I_{m,k_m} \subset [x,y]\mbox{ for }m =1,..., n \mbox{ and } k_m  = 1,..., 2^{m-1}\}
\end{equation*}
and
\begin{equation*}
y' = \sup \{e \in I_{m,k_m}: I_{m,k_m} \subset [x,y]\mbox{ for }m =1,..., n \mbox{ and } k_m  = 1,..., 2^{m-1}\}.
\end{equation*}
If $k_n (x,y)=1,$ by \eqref{definition of f_n} we have that
\begin{equation}\label{case 2-0}
f_{n,s} (x') =f_{n,s} (y').
\end{equation}
If $2^m \le k_n (x,y) \le 2^{m+1} -1$ for $m=1,...,n-1,$ by \eqref{decomposition of [0,1]}, \eqref{estimate of |I_n,k_n|} and \eqref{definition of f_n} we have that
\begin{equation*}
|x-y | \ge \mcl^1 (I_{n-m, k_{n-m}}) \ge C(s) 2^{-j^s _{n-m-1}}
\end{equation*}
and
\begin{equation*}
|f_{n,s} (x') -f_{n,s} (y')|=\frac{2 +...+2^m}{2^n}< 2^{m+1-n}.
\end{equation*}
Whenever $n \ge n^s _0  +1 ,$ it follows from \eqref{j^s _n} that 
\begin{align}\label{case 2}
|f_{n,s} (x') -f_{n,s} (y')| \log^s (|x-y|^{-1})
\le   &  2^{m+1-n} \log^s \left(C^{-1} 2^{j^s _{n-m-1}} \right) \notag \\
\le &C(s)  2^{m+1-n} j^s _{n-m-1}< C(s) .
\end{align}
Notice that there are two cases for the location of $x$
\begin{equation*}
x \in (x' - 2^{-j_n}, x'] ,\ x \in \cup_{m=1}^{n} I_m .
\end{equation*}
If $x \in (x' - 2^{-j_n}, x'],$ by analogous arguments as for \eqref{case 1-2-1} we have that
\begin{equation}\label{case 2-1}
|f_{n,s} (x) -f_{n,s} (x')| \log^s (|x-y|^{-1}) < 1 \qquad \mbox{ whenever } n \ge N (s).
\end{equation}
If $x \in \cup_{m=1}^{n} I_m,$ same arguments as \eqref{case 1-1-3} imply \eqref{case 2-1}.
Analogously, we have that
\begin{equation}\label{case 2-2}
|f_{n,s} (y') -f_{n,s} (y)| \log^s (|x-y|^{-1}) < 1 \qquad \mbox{ whenever } n \ge N (s).
\end{equation}
Since 
\begin{align*}
& |f_{n,s} (x) -f_{n,s} (y)| \log^s (|x-y|^{-1}) \notag\\
=  & \big( |f_{n,s} (x) -f_{n,s} (x')|  + |f_{n,s} (x') -f_{n,s} (y')|  + |f_{n,s} (y') -f_{n,s} (y)| \big) \log^s (|x-y|^{-1}) ,
\end{align*}
by \eqref{case 2-0}, \eqref{case 2}, \eqref{case 2-1} and \eqref{case 2-2} there is a constant $C(s)>0$ such that 
\begin{equation}\label{ii-2}
|f_{n,s} (x) -f_{n,s} (y)| \log^s (|x-y|^{-1}) \le C(s) \qquad 
\end{equation}
whenever $n \ge N(s) \mbox{ and } k_n (x,y) \in \{1,...,  2^n -1\} .$ By \eqref{case 11} and \eqref{ii-2}, we finish the proof of (ii).
\end{proof}

Let $\varphi :\mathbb{S}^1 \rightarrow \mathbb{S}^1$ be a homeomorphism. In the following we denote by $P[\varphi] : \mbd \rightarrow \mbd$ the harmonic extension of $\varphi .$

\begin{example}\label{example 1}
For a given $p \in (1,2),$ there is a homeomorphism $\varphi : \mathbb{S}^1 \rightarrow \mathbb{S}^1$ such that
$\mathcal{V} (p,p-2,0,\varphi) <\infty$, $ I_1 (p,p-2,0,P[\varphi]) =\infty $ and $ I_2 (p,p-2,0,P[\varphi]) =\infty .$ 
\end{example}

\begin{proof}
We first introduce a class of self-homeomorphisms on $\mbs^1$ and their properties.
Let $f_s$ be as in \eqref{definition of f} with $s \in (0,+\infty).$ Define
\begin{equation}\label{g_s}
g_s(x)=\frac{f_s(x)+x}{2} \qquad x \in [0,1] .
\end{equation}
Then $g_s :[0,1] \rightarrow [0,1]$ is strictly increasing and continuous, i.e. $g_s$ is homeomorphic. Moreover by \eqref{property of f}, there is a constant $C(s)>0$ such that
\begin{equation}\label{modulus of continuoty: g}
|g_s (x)-g_s (y)|\le C(s) \log^{-s} (|x-y|^{-1}) 
\end{equation}
for all $x, y \in [0,1] \mbox{ with }x \neq y.$
Let $\mbox{arg} z \in (-\pi,\pi]$ be the principal value of the argument $z .$
Define
\begin{equation}\label{varphi_s}
\varphi_s (z) =  \exp\left(i 2 \pi \left[g_s \left(\frac{\arg z}{2 \pi}  \right)-g_s(\frac{1}{2}) +\frac{1}{2}\right]\right) \qquad z \in \mathbb{S}^1 .
\end{equation}
Then $\varphi _s : \mbs^1 \rightarrow \mbs^1$ is homeomorphic and $\varphi(e^{i \pi}) =e^{i \pi}.$
Next we prove that
\begin{equation}\label{modulus of continuoty: varphi}
|\varphi_s(z_1) - \varphi_s (z_2)| \lesssim \log^{-s} (|z_1 -z_2|^{-1})
\end{equation}
for all $z_1 , z_2 \in \mathbb{S}^1$ with $z_1 \neq z_2 .$
Let $\Gamma(z_1, z_2)$ be the arc in $\mathbb{S}^1$ joining $z_1$ to $z_2$ with smaller length.
Denote by $\ell(\Gamma(z_1, z_2))$ the length of $\Gamma(z_1, z_2) .$
In order to prove \eqref{modulus of continuoty: varphi}, it is enough to consider the case $\ell(\Gamma(z_1, z_2)) \ll 1.$ 
If $e^{i \pi} \notin \Gamma(z_1, z_2),$ we have that
\begin{equation*}
|\arg z_1 - \arg z_2| \approx |z_1 -z_2| \mbox{ and } \big|g_s\left(\frac{\arg z_1}{2 \pi}  \right) - g_s \left(\frac{\arg z_2}{2 \pi}  \right) \big| \approx |\varphi_s (z_1) -\varphi_s (z_2)| 
\end{equation*}
whenever $\ell(\Gamma(z_1, z_2)) \ll 1.$
Together with \eqref{modulus of continuoty: g}, we then have that
\begin{align}\label{modulus of continuoty: varphi-1}
|\varphi_s (z_1) - \varphi_s (z_2)| \approx & \big|g_s\left(\frac{\arg z_1}{2 \pi}  \right) - g_s \left(\frac{\arg z_2}{2 \pi}  \right)\big| \notag \\
\lesssim & \log^{-s} (|\arg z_1- \arg z_2|^{-1}) \approx \log^{-s} (| z_1 - z_2|^{-1}).
\end{align}
If $e^{i \pi} \in \Gamma(z_1, z_2)$ and $\ell(\Gamma(\varphi(z_1), \varphi(e^{i \pi}))) > \ell(\Gamma(\varphi(e^{i \pi}), \varphi(z_2))),$ there is $z_0 \in \Gamma(z_1, e^{i \pi})$ such that 
\begin{equation}\label{modulus of continuoty: varphi-2}
|\varphi_s(z_1) -\varphi_s (z_2)| \lesssim |\varphi_s (z_1) -\varphi_s (z_0)|.
\end{equation}
Same arguments as for \eqref{modulus of continuoty: varphi-1} imply that
\begin{equation}\label{modulus of continuoty: varphi-3}
|\varphi_s (z_1) -\varphi_s (z_0)| \lesssim \log^{-s} (| z_1 - z_0|^{-1}) \lesssim  \log^{-s} (| z_1 - z_2|^{-1}).
\end{equation}
Combining \eqref{modulus of continuoty: varphi-2} with \eqref{modulus of continuoty: varphi-3} therefore implies that 
\eqref{modulus of continuoty: varphi} holds when $e^{i \pi} \in \Gamma(z_1, z_2)$ and $\ell(\Gamma(\varphi_s (z_1), \varphi_s (e^{i \pi}))) > \ell(\Gamma(\varphi(e^{i \pi}).$ 
Analogously, we can prove that \eqref{modulus of continuoty: varphi} holds when $e^{i \pi} \in \Gamma(z_1, z_2)$ and $\ell(\Gamma(\varphi_s (z_1), \varphi_s (e^{i \pi}))) \le \ell(\Gamma(\varphi_s (e^{i \pi}), \varphi_s (z_2))).$ 

Let $p \in (1,2).$ There is $s \in (1,+\infty)$ such that $p-1 < 1/s <1.$
Based on this $s,$ we obtain a homeomorphism $\varphi= \varphi_s : \mathbb{S}^1 \rightarrow \mathbb{S}^1 ,$ where $\varphi_s$ is from \eqref{varphi_s}.  
By Jensen's inequality and \eqref{modulus of continuoty: varphi}, we have that
\begin{align*}
\mathcal{V} (p,p-2,0,\varphi)=& \int_{\mathbb{S}^1} \big( \int_{\mathbb{S}^1} \log|\varphi^{-1}(\xi) - \varphi^{-1}(\eta) |^{-1}  \, |d \eta| \big)^{p-1} \, |d \xi|  \notag \\
\le & \big( \int_{\mathbb{S}^1}  \int_{\mathbb{S}^1} \log |\varphi^{-1}(\xi) - \varphi^{-1}(\eta) |^{-1} \, |d \eta| \, |d \xi| \big)^{p-1} \\
\lesssim & \big( \int_{\mathbb{S}^1}  \int_{\mathbb{S}^1} |\xi -\eta|^{-\frac{1}{s}} \, |d \eta| \,  |d \xi| \big)^{p-1} < +\infty.
\end{align*}
Let $n^s _0$ be as in \eqref{n_0} with $s$ chosen above.
For any $n \ge n^s _0$ and any $ j_n < j \le j_{n+1},$ by \eqref{g_s} and \eqref{definition of f} we have that
\begin{align}\label{example 1-1}
\sum_{k=1} ^{2^j} \ell(\varphi(\Gamma_{j,k}))^p
= & 2 \pi \sum_{k=1} ^{2^j} \mcl^1(g_s([(k-1)2^{-j}, k 2^{-j}]))^p \notag \\
\gtrsim &  \sum_{k=1} ^{2^j} (f_s(k 2^{-j}) - f_s((k-1)2^{-j}))^p = 2^{(n+1) (1-p)}.
\end{align}
Notice that $j_{n+1} -j_{n} \approx 2^{n/s}$ whenever $n \ge n_0 .$
We then derive from \eqref{example 1-1} that
\begin{equation*}
\mathcal{E}_1 (p,p-2,0,\varphi)
\ge  \sum_{n=n_0} ^{+\infty} \sum_{j_n < j \le j_{n+1}}  \sum_{k=1}^{2^j} \ell(\varphi(\Gamma_{j,k}))^p
\gtrsim \sum_{n=n_0} ^{+\infty} 2^{n(1-p+\frac{1}{s})} = +\infty.
\end{equation*}
By Lemma \ref{connect mathcal E_1 and mathcal E_2}, Lemma \ref{I_1 discrete} and Lemma \ref{desvribe I_2}, it follows that 
$I_1 (p,p-2,0,P[\varphi]) =\infty $ and $ I_2 (p,p-2,0,P[\varphi]) =\infty .$
\end{proof}

\begin{example}\label{example 2}
For a given $p \in (2,+\infty),$ there is a homeomorphism $\varphi : \mathbb{S}^1 \rightarrow \mathbb{S}^1$ such that
$\mathcal{V} (p,p-2,0,\varphi) =\infty$, $I_1 (p,p-2,0,P[\varphi]) <\infty $ and $ I_2 (p,p-2,0,P[\varphi]) <\infty .$

\end{example}

\begin{proof}
Since $p \in (2,+\infty),$ there is $s \in (0,1)$ such that $p-1 > 1/s >1.$
Based on this chosen $s,$ we obtain a homeomorphism $\varphi= \varphi_s : \mathbb{S}^1 \rightarrow \mathbb{S}^1 ,$ where $\varphi_s$ is from \eqref{varphi_s}.
In order to prove $\mathcal{V} (p,p-2,0,\varphi) =\infty$, by Jensen's inequality it suffices to prove that
\begin{equation}\label{example 2-0}
\int_{\mathbb{S}^1} \int_{\mathbb{S}^1} \log |\varphi^{-1}(\xi) - \varphi^{-1}(\eta)|^{-1} \, |d \eta| \,  |d \xi| =+\infty.
\end{equation}
For any $\sigma \in \mbs^1$ and $\tau \in \mbs^1 ,$ let $\ell (\sigma ,\tau)$ be the arc length of the shorter arc in $\mbs^1$ joining $\sigma$ and $\tau .$
Let $n^s _0 $ be from \eqref{n_0} with $s$ chosen above.
For any $n \ge n^s _0 ,$ set
\begin{equation*}
\Gamma_n = \{(\xi, \eta) \in \mathbb{S}^1 \times \mathbb{S}^1:\pi 2^{1- j_{n+1}} < \ell(\varphi^{-1}(\xi), \varphi^{-1}(\eta)) \le \pi 2^{1-j_{n}}  \}.
\end{equation*}
We have that
\begin{align}\label{example 2-1}
\int_{\mathbb{S}^1} \int_{\mathbb{S}^1} \log|\varphi^{-1}(\xi) - \varphi^{-1}(\eta)|^{-1} \, |d \eta| \, |d \xi|
\ge & \sum_{n=n_0 } ^{+\infty} \int_{\Gamma _n} \log |\varphi^{-1}(\xi) - \varphi^{-1}(\eta)|^{-1} \, |d \eta| \, |d \xi| \notag \\
\gtrsim & \sum_{n=n_0 } ^{+\infty} j_n \int_{\Gamma _n} \, |d \eta| \, |d \xi|.
\end{align}
Given $n \ge n^s _0$ and $k=1,..., 2^n,$ let  
\begin{equation*}
\Gamma^{'} _{n,k} = \exp (i 2\pi [b_{n,2k-1},2^{-j_{n+1}}+b_{n,2k-1}]),
\end{equation*}
\begin{equation*}
\Gamma^{''} _{n,k} = \exp (i 2\pi [2^{-j_n}-2^{-j_{n+1}}+ b_{n,2k-1},2^{-j_n} +b_{n,2k-1}]) .
\end{equation*}
For any $\xi \in \varphi(\Gamma^{'} _{n,k})$ and $\eta \in \varphi(\Gamma^{''} _{n,k}),$
we have that
\begin{equation}\label{nokam1}
2\pi (2^{-j_n} -2^{1-j_{n+1}}) \le \ell(\varphi^{-1}(\xi), \varphi^{-1}(\eta)) \le \pi \cdot 2^{1-j_n}.
\end{equation}
Notice that by \eqref{n_0} we have that $ 2^{-j_{n+1}} < 2^{-j_n} -2^{1-j_{n+1}}$ whenever $n \ge n^s _0 $.
It then follows from \eqref{nokam1} that
\begin{equation}\label{example 2-2}
\varphi(\Gamma^{'} _{n,k}) \times \varphi(\Gamma^{''} _{n,k}) \subset \Gamma_n
\end{equation}
for all $n \ge n^s _0$ and all $k=1,..., 2^n.$
Moreover from \eqref{varphi_s}, \eqref{g_s} and \eqref{definition of f}, it follows that
\begin{align}\label{example 2-3}
\ell(\varphi(\Gamma^{'} _{n,k})) = & 2\pi \mcl^1 (g([b_{n,2k-1},2^{-j_{n+1}} + b_{n,2k-1}] ))  \notag \\
\ge & \pi ( f_s (2^{-j_{n+1}}) -f_s (0) ) =  \pi 2^{-n-1}.
\end{align}
for all $n \ge n^s _0$ and all $k=1,..., 2^n.$
Similarly 
\begin{equation}\label{example 2-3-1}
\ell(\varphi(\Gamma^{''} _{n, k})) \ge \pi 2^{-n-1} .
\end{equation}
Since $(\varphi(\Gamma^{'} _{n,k}) \times \varphi(\Gamma^{''} _{n,k})) \cap (\varphi(\Gamma^{'} _{n,j}) \times \varphi(\Gamma^{''} _{n,j}) ) = \emptyset$ for all $n \ge n^s _0$ and $k,j \in \{1,..., 2^n\}$ with $k \neq j,$
it follows \eqref{example 2-2}, \eqref{example 2-3} and \eqref{example 2-3-1} that
\begin{equation}\label{example 2-3-2}
\int_{\Gamma _n} \, |d \eta| \, |d \xi| \ge \sum_{k=1} ^{2^n} \int_{\varphi(\Gamma^{'} _{n,k}) \times \varphi(\Gamma^{''} _{n,k})} \, |d \xi| \, |d \eta| \ge \pi^2 2^{-n-2}
\end{equation}
for all $n \ge n^s _0.$ Combining \eqref{example 2-1} with \eqref{example 2-3-2} hence implies that
\begin{equation*}
\int_{\mathbb{S}^1} \int_{\mathbb{S}^1} \log |\varphi^{-1}(\xi) - \varphi^{-1}(\eta)|^{-1} \, |d \eta| \, |d \xi|
\gtrsim \sum_{n=n_0} ^{+\infty} \frac{j_n}{2^n}
\approx \sum_{n=n_0} ^{+\infty} \frac{2^{\frac{n}{s}}}{2^n} =+\infty.
\end{equation*}
Therefore \eqref{example 2-0} is complete.

For any $n \ge n_0 $ and $j_n < j \le j_{n+1},$ by \eqref{varphi_s}, \eqref{g_s}, \eqref{definition of f} and Jensen's inequality we have that
\begin{align}\label{example 2-4}
\sum_{k=1} ^{2^j} \ell (\varphi(\Gamma_{j,k}))^p
= &2 \pi \sum_{k=1} ^{2^j} \mcl^1 (g_s ([(k-1)2^{-j}, k 2^{-j}]))^p \notag\\
\lesssim & \sum_{k=1} ^{2^j} (f_s(k 2^{-j})-f_s((k-1)2^{-j}))^p + \sum_{k=1} ^{2^j} 2^{-pj} \notag\\
= & 2^{(1-p)(n+1)} + 2^{(1-p)j}
\end{align}
Notice $j_{n+1} -j_n \approx 2^{n/s}$ whenever $n \ge n^s _0 .$ We then derive from \eqref{example 2-4} that
\begin{align}\label{example 2-5}
\sum_{j=j_{n_0 } +1} ^{+\infty} \sum_{k=1} ^{2^j} \ell (\varphi(\Gamma_{j,k}))^p
\lesssim & \sum_{n=n_0 } ^{+\infty} \sum_{j_n <j \le j_{n+1}} 2^{(1-p)(n+1)} + \sum_{j=j_{n_0 } +1} ^{+\infty} 2^{(1-p)j} \notag \\
\approx & \sum_{n=n_0 } ^{+\infty} 2^{n(1-p+\frac{1}{s})} + \sum_{j=j_{n_0 } +1} ^{+\infty} 2^{(1-p)j} <+\infty.
\end{align}
Since $\sum_{j= 1} ^{j_{n_0 }} \sum_{k=1} ^{2^j} \ell (\varphi(\Gamma_{j,k}))^p $ is finite, it follows from \eqref{example 2-5} that $\mathcal{E}_1 (p,p-2,0,\varphi) <+\infty.$ Moreover by Lemma \ref{I_1 discrete} and Lemma \ref{desvribe I_2} we have that 
$I_1 (p,p-2,0,P[\varphi]) <+\infty $ and $ I_2 (p,p-2,0,P[\varphi]) <+\infty .$
\end{proof}

\begin{example}\label{reasonable example}
There is a homeomorphism $\varphi : \mbs^1 \rightarrow \mbs^1$ such that both $I_1 (p,\alpha ,\lambda ,P[\varphi])<+\infty$ and $I_2 (p,\alpha ,\lambda ,P[\varphi])<+\infty$ hold for 
all $p>1,$ $\alpha \in (-1, p-1)$ and $\lambda \in \mathbb{R}.$
Moreover for any $p>1,$ there is a homeomorphism $\varphi : \mbs^1 \rightarrow \mbs^1$ such that
$I_1 (p,\alpha ,\lambda ,p[\varphi])=\infty$ and $I_2 (p,\alpha ,\lambda ,P[\varphi])=\infty$ whenever either $\alpha \in (-1, p-2)$ and $\lambda \in \mathbb{R}$ or $\alpha=p-2$ and $\lambda \in [-1, +\infty).$
\end{example}

\begin{proof}
Take $\varphi : \mbs^1 \rightarrow \mbs^1$ as the identity mapping. We have that
\begin{equation*}
\mathcal{E}_1 (p,\alpha,\lambda,\varphi) 
\approx \sum_{j=1} ^{+\infty} 2^{j(p-2-\alpha)} j^{\lambda} 2^j (2^{1-j} \pi)^p
\approx \sum_{j=1} ^{+\infty} 2^{-j(1+\alpha)} j^{\lambda} < +\infty
\end{equation*}
whenever $p>1$, $\alpha \in (-1, p-1)$ and $\lambda \in \mathbb{R}.$
Therefore by Lemma \ref{I_1 discrete} and Lemma \ref{desvribe I_2} both $I_1 (p,\alpha ,\lambda ,P[\varphi])$ and $I_2 (p,\alpha ,\lambda ,P[\varphi])$ are finite now.

For a given $p >1,$ set $j_n$ in \eqref{j^s _n} as $[e^{2^{n(p-1)}}].$ 
There is $n_0 \ge 1$ such that \eqref{n_0} holds for all $n \ge n_0 -1.$ 
By following the arguments for \eqref{decomposition of [0,1]}, we have $f$ as in \eqref{definition of f}.
Moreover by same arguments as in the proof of Proposition \ref{property of f}, there is a constant $C>0$ depending only on $p$ such that  
\begin{equation*}
|f(x) -f(y)| \log^{\frac{1}{p-1}} \log (|x-y|^{-1}) \le C
\end{equation*}
for all $x, y \in [0,1]$ with $x \neq y.$
As in \eqref{varphi_s}, we finally obtain a homeomorphism $\varphi : \mbs^1 \rightarrow \mbs^1 .$
For any $n \ge n_0$ and $j_n < j \le j_{n+1},$ by analogous arguments for \eqref{example 1-1} we have that
\begin{equation}\label{reasonable example: 0}
\sum_{k=1} ^{2^j} \ell(\varphi(\Gamma_{j,k}))^p
\gtrsim 2^{n(1-p)}.
\end{equation}
Notice that $\sum_{j_n <j \le j_{n+1}} j^{-1}
\approx \log j_{n+1} - \log j_n \gtrsim 2^{n(p-1)} $ for all $n \ge n_0 .$
For any $\lambda \in [-1,+\infty)$ it then follows from \eqref{reasonable example: 0} that
\begin{align}\label{reasonable example: 1}
\mathcal{E}_1 (p, p-2 , \lambda, \varphi)
\ge&   \sum_{j=1} ^{+\infty} \sum_{k=1}^{2^j} \ell(\varphi(\Gamma_{j,k}))^p j^{-1} 
\ge \sum_{n=n_0} ^{+\infty} \sum_{j_n < j \le j_{n+1}} j^{-1} \sum_{k=1}^{2^j} \ell(\varphi(\Gamma_{j,k}))^p \notag \\
\gtrsim & \sum_{n=n_0} ^{+\infty} 2^{n(p-1)} \cdot 2^{n(1-p)} =+\infty.
\end{align}
For any $\alpha \in (-1,p-2)$ and $\lambda \in \mathbb{R},$ we have that $2^{j(p-2-\alpha)} j^{\lambda} \gtrsim j^{-1}$ whenever $j \gg 1.$
Without loss of generality, we assume that $2^{j(p-2-\alpha)} j^{\lambda} \gtrsim j^{-1}$ for all $n \ge n_0$ and $j_n < j \le j_{n+1}.$
Hence from \eqref{reasonable example: 1} we have that
\begin{align}\label{reasonable example: 2}
\mathcal{E}_1 (p,\alpha,\lambda,\varphi) \ge  & \sum_{n=n_0} ^{+\infty} \sum_{j_n < j \le j_{n+1}} \sum_{k=1}^{2^j} \ell(\varphi(\Gamma_{j,k}))^p 2^{j(p-2-\alpha)} j^{\lambda}  \notag \\
\gtrsim & \sum_{n=n_0} ^{+\infty} \sum_{j_n < j \le j_{n+1}} \frac{1}{j} \sum_{k=1}^{2^j} \ell(\varphi(\Gamma_{j,k}))^p
 = +\infty
\end{align}
for all $\alpha \in (-1,p-2)$ and $\lambda \in \mathbb{R} .$
By Lemma \ref{connect mathcal E_1 and mathcal E_2}, Lemma \ref{I_1 discrete} and Lemma \ref{desvribe I_2}, we conclude from \eqref{reasonable example: 1} and \eqref{reasonable example: 2} that for any $p>1$
there is a homeomorphism $\varphi: \mathbb{S}^1 \rightarrow \mathbb{S}^1$ such that $I_1 (p,\alpha ,\lambda ,P[\varphi]) =\infty$ and $I_2 (p,\alpha ,\lambda ,P[\varphi]) =\infty$ whenever either $\alpha \in (-1, p-2)$ and $\lambda \in \mathbb{R}$ or $\alpha=p-2$ and $\lambda \in [-1, +\infty).$
\end{proof}

\section*{Acknowledgment}
The author has been supported by China Scholarship Council (project No. 201706340060). 
This paper is a part of the author's doctoral thesis.
The author thanks his advisor Professor Pekka Koskela for posing this question and for valuable discussions.
The author thanks Aleksis Koski and Zhuang Wang for comments on the earlier draft.

\bigskip
\bibliographystyle{amsplain}

\noindent Haiqing Xu

\noindent 
Department of Mathematics and Statistics, University of Jyv\"askyl\"a, PO~Box~35, FI-40014 Jyv\"askyl\"a, Finland

\noindent 
School of Mathematical Sciences, University of Science and Technology of China, Hefei 230026, P. R. China

\noindent{\it E-mail address}:  \texttt{hqxu@mail.ustc.edu.cn}

\end{document}